\documentclass[reqno,11pt]{amsart}
\usepackage{amssymb,amsmath,amsthm,amstext,amsfonts}
\usepackage[dvips]{graphicx}
\usepackage{a4wide}
\usepackage{psfrag}
\usepackage{color}

\makeatletter \@addtoreset{equation}{section} \makeatother

\renewcommand\thetable{\thesection.\@arabic\c@table}

\theoremstyle{plain}
\newtheorem{maintheorem}{Theorem}

\newtheorem{maincorollary}{Corollary}

\newtheorem{theorem}{Theorem }[section]
\newtheorem{proposition}[theorem]{Proposition}
\newtheorem{lemma}[theorem]{Lemma}
\newtheorem{corollary}[theorem]{Corollary}

\theoremstyle{definition} \theoremstyle{remark}
\newtheorem{remark}[theorem]{Remark}

\newtheorem{definition}[theorem]{Definition}

\newcommand{\vep}{\varepsilon}

\newcommand{\la} {\lambda}

\newcommand{\C}{\mathbb{C}}

\newcommand{\N}{\mathbb{N}}
\newcommand{\R}{\mathbb{R}}

\newcommand{\topp}{\operatorname{top}}

\newcommand{\vr}{\varphi}

\newcommand{\Ptop}{P_{\topp}}

\newcommand{\cL}{\mathcal{L}}
\newcommand{\cE}{\mathcal{E}}

\def\ds{\displaystyle}

\begin{document}

\title{Linear response, and consequences for differentiability of statistical quantities and Multifractal Analysis }

\author{ Thiago Bomfim and Armando Castro}

\address{Thiago Bomfim, Departamento de Matem\'atica, Universidade Federal da Bahia\\
Av. Ademar de Barros s/n, 40170-110 Salvador, Brazil.}
\email{tbnunes@ufba.br}
\urladdr{https://sites.google.com/site/homepageofthiagobomfim/}

\address{Armando Castro, Departamento de Matem\'atica, Universidade Federal da Bahia\\
Av. Ademar de Barros s/n, 40170-110 Salvador, Brazil.}
\email{armandomath@gmail.com}

\date{\today}

\begin{abstract}
In this article we  initially prove the differentiability of the topological pressure, equilibrium states and
their densities with respect to  smooth expanding dynamical systems and any smooth potential. This is done by proving the regularity of the dominant eigenvalue of the transfer operator with respect to dynamics and potential.
From that, we obtain strong consequences on the regularity of the dynamical system statistical properties, that apply in more general contexts.
Indeed, we prove that the average and variance obtained from the Central Limit
Theorem vary $C^{r-1}$  with respect to the $C^{r}-$expanding dynamics and $C^{r}-$potential, and also, there is a large
deviations principle exhibiting  a $C^{r-1}$ rate with respect to the
dynamics and the potential. An application for multifractal analysis is given. We also obtained  asymptotic formulas for the derivatives of the topological pressure and other thermodynamical quantities.
\end{abstract}


\maketitle

\section{Introduction}

The thermodynamical formalism in dynamical systems was initially concerned with the existence and uniqueness of equilibrium states for a family of dynamics and potentials.  Since we obtain uniqueness of equilibrium states for  robust families of dynamics and potentials it is natural to wonder if the topological pressure and respective equilibrium states vary continuously with respect to the dynamic and  potential under deterministic
 perturbations. This type of question is called statistical stability and has its importance both in mathematics and in physics, since very often, the dynamical system that describes a physical situation has a deterministic or random error.

When one obtains results of statistical stability the next natural step is to ask about  highest regularity results, particularly results on the differentiability of thermodynamic objects with respect to the dynamics and the  potential. In such direction, explicit formulas for the derivatives are also welcome.
These differentiability results have been referred as linear response formulas (see e.g. \cite{Rue09}). This has proved to be a hard subject not yet completely understood. In fact,
linear response formulas have been obtained mostly for uniformly hyperbolic diffeomorphisms and flows in \cite{KKPW89,Ji12,Rue97,BL07,GL06}, for the SRB measure of some partially hyperbolic difeomorphisms in \cite{Dol04} and for one-dimensional piecewise expanding and quadratic maps in \cite{Rue05,BS08,BS09,BS12,BBS14}. More recently, \cite{GP17} studied the differentiable dependence of densities of SRB measures for expanding maps in the circle.  The general picture is still far from complete, even in the expanding case, and for any equilibrium states.

In article \cite{BCV16} linear response formulas, continuity and differentiability of several thermodynamical quantities and limit theorems of a robust non-uniformly expanding setting of \cite{VV10,CV13} were obtained. In such setting,  the coexistence of expanding and contracting behavior was admitted, but the potential should be close to a constant function.

In the specific context of expanding dynamics and Holder potentials It is known that there is one unique equilibrium state such that the statistical stability holds (see e.g. \cite{PU10}). 
The linear response results in \cite{BCV16} also hold for any smooth expanding dynamics. But such results only deal with smooth potentials close to a constant function. 
In this paper, we have two main goals.

The first goal is to obtain linear response formulas for any equilibrium state $\mu_{f, \phi}$ associated to a smooth expanding dynamics $f$ and any smooth potential $\phi$. This  is a kind of extension of \cite{BCV16} for \emph{all} smooth potentials, in the context of smooth expanding dynamics. Let us mention that a similar result of this goal has parallelly obtained through other (Lasota-Yorke) techniques in \cite{Se17}. We follow here another approach, using Cones and Projective metrics theory.   
The Birkhoff Cones approach has great advantages for further generalizations. This is because the fact that Lasota-Yorke techniques 
need aditional (hybrid) 
 topological mixing assumptions in order to prove that we 
have a simple dominant eigenvalue for the transfer operator. 
These hybrid assumptions are not necessary while using Birkhoff Cones. For instance, examples in \cite{CV13, BCV16} include bifurcations of Manneville-Poumeau map in which a sink is created, and hence, the perturbed system is not topologically mixing, but Cone techniques naturally apply, and as a consequence, a unique fast mixing equilibrium measure exists for low variation potentials. See also \cite{CN17}, for more examples, in the context of nonhyperbolic diffeomorphisms, where the Cones techniques also were used to study statistical properties of equilibrium states and a topological mixing condition is not assumed a priori. In fact, in that context sinks may also 
appear as we bifurcate a hyperbolic attractor.
Our statements also include some higher order results.

Our second main goal is to
derive the stability of statistical law rates related to the Central Limit Theorem and Large Deviations principle from the linear response formula.
For that, we just use the linear response formula results in a abstract way, so, it is not necessary that
the dynamics is uniformly expanding.

Finally, we apply the stability of statistical quantities involved in the Large Deviations principle to  Multifractal Analysis.
Even though we use specification property, which is peculiar to the expanding dynamics setting, we expect
these last result also to apply to contexts of robust
nonuniformly specification.

The strategy  to obtain our results begins  by studying the spectra of transfer operators associated to the equilibrium measures.
We prove that such operators have the spectral gap property on the smooth functions space. For this result, we use the technique of cones and projective metrics.
Moreover, the spectral gap property occurs uniformly for dynamics and potential that are sufficiently close. However, the dependence of the operator
with respect to the dynamics is not continuous in the norm operator topology, and so, It is not possible to use the classical perturbative spectral theory.
 Nevertheless, the uniformity of the spectral gap enables us to employ a perturbative framework developed in
\cite{GL06} (see also \cite{GL08}), to prove that the topological pressure and the equilibrium states varies $C^{r-1}$ with
respect to $C^{r}$ expanding dynamics and potentials.
Using some ideas present in \cite{BCV16}, we obtain asymptotic formulas for the first derivative of the topological pressure.
As a consequence of these  regularity results, we obtain higher regularity statements for  certain thermodynamic quantities,
limit theorems and  rate of large deviations. Lastly, using \cite{BV15} we apply the results above to the study of the multifractal analysis
of smooth expanding dynamics.

This article is organized as follows. In Section~\ref{Statement of the main results} we provide some definitions
and the statement of  the main results. In Section~\ref{prelim} we recall the necessary framework on cones and projective metrics, and we also prove a uniform spectral gap  property for any transfer operator associated to smooth expanding dynamics and potentials. Section~\ref{Proofs} is devoted to the proof of the main results.
%
%

\section{Statement of the main results}\label{Statement of the main results}

Along the text,  $M$ will  always denote a connected and compact Riemannian manifold, with Riemannian distance $d$.
We denote by $\mathcal{D}^{r}$, $r \geq 1$, the space of $C^{r}$ maps $f : M \rightarrow M$
such that $f$ is expanding. For an expanding map $f$ we mean that there exists $\sigma> 1$ such that  $\sigma < [\sup_{x \in M}\{\|[Df(x)]^{-1}\|\}]^{-1}$. Note that given  an expanding map $f$, there one can take $\sigma$ to be the same for all map in  a sufficiently small neighborhood  of $f$ in $\mathcal{D}^r, r \geq 1$.
Given an integer $r \geq 1$, we will denote by $C^{r}(M , \R)$ the  Banach space of  $C^r$ functions whose domain is $M$.

\begin{remark}
The transfer operator $\mathcal{L}_{f , \phi}$ acting on a function $g : M \rightarrow \C$ is defined as the function
$\mathcal{L}_{f , \phi}g(x) := \sum_{y \in f^{-1}(x)}e^{\phi(y)}g(y)$.
\end{remark}

Let us  recall some necessary definitions. 
Given a continuous map $f:M\to M$ and a potential $\phi:M \to \mathbb R$,
the variational principle for the pressure asserts that
\begin{equation*}
\label{variational principle} \Ptop(f,\phi)=\sup \Big\{
h_\mu(f)+\int \phi \;d\mu : \mu \;\text{is}\; f\text{-invariant}
\Big\}
\end{equation*}
where $\Ptop(f,\phi)$ denotes the topological pressure of $f$ with
respect to $\phi$ and $h_\mu(f)$ denotes the metric entropy. An
\textit{equilibrium state} for $f$ with respect to $\phi$ is an
invariant measure that attains the supremum in the right hand side
above.

A probability measure $\nu$, not necessarily invariant, is
\emph{conformal} if there exists a function $\psi:M\to \R$ such
that $\nu(f(A))=\int_A e^{-\psi} d\nu$ 
for every measurable set $A$ such that $f \mid A$ is injective.
Let $S_n \phi=\sum_{j=0}^{n-1} \phi \circ f^j$ denote the $n$th
Birkhoff sum of a function $\phi$.

It follows from \cite{ruelle89} that given $f \in D^{r}$ and
 a  Holder continuous potential $\phi : M \rightarrow \R$,
 there exists an unique equilibrium state $\mu_{f , \phi}$ for $f$
 with respect to $\phi$. Furthermore, the support of $\mu_{f,\phi}$ is $M$.
  Moreover, $\mu_{f , \phi } = h_{f , \phi}\nu_{f, \phi}$, where
$\mathcal{L}_{f , \phi} h_{f , \phi} = e^{P_{\topp}(f , \phi)}
h_{f , \phi}\;$, $\;\mathcal{L}_{f , \phi}^{\ast} \nu_{f , \phi} =
e^{P_{\topp}(f , \phi)} \nu_{f , \phi}\;$ and $\;\lambda_{f , \phi}
:= e^{P_{\topp}(f , \phi)}$ is the spectral radius of the transfer operator
$\mathcal{L}_{f , \phi}$ acting on $C^{0}(M , \R)$.

\subsection{Linear response formula}

It is well known that the equilibrium state $\mu_{f,\phi}$ and the topological pressure $P_{\topp}(f , \phi)$ depend continuously on the expanding dynamics and are smooth with respect to the potential (see for example \cite{PU10}).
In this paper we will show that these objects are also differentiable with respect to the dynamics $f$. Let us explain what we mean by differentiable dependence, which is the same as the one exposed in section 3 of \cite{BCV16}. 

Recall that $C^r(M,M)$ is a submanifold of
$C^r(M, \mathbb R^m)$ and that for any $f \in C^{r}(M , M)$ the tangent space
$T_{f} C^{r}(M , M)$ is naturally identified with the space
\begin{equation}\label{eq:spaceG}
\Gamma^{r}_{f} := \{H \in C^{r}(M , TM) : H(x) \in T_{f(x)}M, \forall x \in M\}
\end{equation}
of $C^r$-sections (or vector fields) over $f$. The space $\Gamma^{r}_{f}$ is Banachable, that is,
since it is naturally isomorphic to
\begin{equation}\label{eq:spaceGB}
E^{t} = \{H \in C^{r}(M , \mathbb R^{m}) : (f(x), H(x)) \in TM \},
\end{equation}
then it inherits a structure of Banach space.
One should mention the later identification is independent of the embedding of $M$ (cf. page 238 in \cite{Fr79}).
Throughout we will consider the space $\text{Diff}_{\text{loc}}^r(M)$
as a Banach manifold modeled by $C^r(M, \mathbb R^m)$,
from which $T_{f} \text{Diff}_{\text{loc}}^r(M) \simeq E^t \subset C^r(M, \mathbb R^m)$ for every ${f} \in \text{Diff}_{\text{loc}}^r(M)$. Therefore, It make sense 
to consider Fr\'echet (classical) derivatives with respect to the dynamics $f \in C^{r}(M , M)$:

\begin{definition}
Let $F$ be a Banach space. Given $f_0 \in \text{Diff}_{\text{loc}}^r(M)$, we say that a function $\Psi : \text{Diff}_{\text{loc}}^r(M) \subset C^r(M, \mathbb R^m) \to F$
is \emph{Fr\'echet differentiable at $f_0$ } if there exists a continuous linear map
$D\Psi(f_0) : T_{f_0} \text{Diff}_{\text{loc}}^r(M) \subset C^r(M, \mathbb R^m) \to  F$ so that
\begin{equation}\label{eq:deriv}
\lim_{H \to 0} \frac1{\|H\|_{C^r(M, \mathbb R^m)}} \; \| \Psi(f_0 + H) -\Psi (f_0) - D\Psi(f_0) H\|_F=0
\end{equation}
where $H$ is taken to converge to zero in $T_{f_0} \text{Diff}_{\text{loc}}^r(M) = T_{f_0} C^{r}(M , M)$. The functional $D\Psi(f_0)$ is
the derivative of $\Psi$ at $f_0$. We say that $\Psi$ is $C^1$-\emph{differentiable} if the Fr\'echet derivative $D\Psi(f)$ exists at every $f \in \text{Diff}_{\text{loc}}^r(M)$ and the map $\text{Diff}_{\text{loc}}^r(M) \ni f \mapsto D\Psi(f) \in B(C^r(M, \mathbb R^m), F)$
is continuous.
\end{definition}

For notational simplicity, when no confusion is possible we shall omit the
spaces $C^r(M, \mathbb R^m)$ and $F$ in the norms whenever using the expression~\eqref{eq:deriv} to compute derivatives of vector valued maps.

Our first main result assures that the topological pressure and equilibrium state varies smoothly with respect
to the smooth expanding dynamics.

\begin{maintheorem}\label{thm:linres}
If $r \geq 2$ then the following maps are $C^{r-1}$:
{
 \begin{itemize}
  \item[i.] $\mathcal{D}^{r} \times C^{r}(M , \R) \ni (f ,\phi) \mapsto P_{\topp}(f,\phi)$;
  \item[ii.] $\mathcal{D}^{r} \times C^{r}(M , \R) \ni (f ,\phi)  \mapsto h_{f,\phi} \in \mathcal{C}^{0}$;
  \item[iii.] $\mathcal{D}^{r} \times C^{r}(M , \R) \ni (f ,\phi)  \mapsto \nu_{f,\phi} \in [C^{r}]^{\ast}$;
  \item[iv.] $\mathcal{D}^{r} \times C^{r}(M , \R) \ni (f ,\phi)  \mapsto \mu_{f,\phi} \in [C^{r}]^{\ast}$.
 \end{itemize}
 }
Moreover, $D_{f ,\phi}P_{\topp}(f , \phi)$ acting on $( H_{1} , H_{2})$ is given by
$$
\lambda_{f,\phi}^{-1}\sum_{j = 1}^{\deg(f)}  \int e^{\phi(f_{j}(\cdot))}
Dh_{f , \phi|f_{j}(\cdot)}  \odot   [(T_{j|f} \odot H_{1})(\cdot)] \; d\nu_{f , \phi} \;+
$$
$$
 \lambda_{f,\phi}^{-1}\sum_{j =
1}^{deg(f)} \int e^{\phi(f_{j}(\cdot))}    h_{f , \phi}(f_{j}(\cdot))    D\phi_{|f_{j}         (\cdot)} \odot
[(T_{j|f} \odot H_{1})(f_{j}(\cdot))] \;d\nu_{f} + \int H_{2} \; d\mu_{f,\phi},
$$
where $f_{j}$ denote the local inverse branches of $f$ and
 $T_{j|f} \odot H_{1}$ is the derivative of $f \mapsto f_{j}$ acting on $H_{1}$.
\end{maintheorem}

The proof of the previous result follows from the  spectral gap uniformity
 of the transfer operators acting on $C^{r}(M , \C)$. This uniformity of the  spectral gap follows
 from cones and projective metrics techniques. Note that the differentiability
 of the equilibrium state as a functional acting in $C^{1}(M ,\R)$ is stronger than a pointwise linear response formula result.

\begin{remark}
It follows from \cite{BCV16} a explicit formula for the first order derivative of the maximal entropy measure with respect  to the smooth expanding dynamics.
\end{remark}

The previous theorem implies the following two corollaries:

\begin{maincorollary}\label{cor1}
Let $r \geq 3$. The following functions are $C^{r-2}$:
\begin{itemize}
    \item[i.] $\mathcal{D}^{r} \times \R \ni  (f,t) \mapsto P_{\topp}(f , t\log |\det Df|)$;
    \item[ii.] $\mathcal{D}^{r} \times \R \ni  (f,t) \mapsto P_{\topp}(f , t\log || Df^{\pm1}||)$;
    \item[iii.] $\mathcal{D}^{r}\times C^{r}(M , \R) \ni (f , \phi) \mapsto h_{\mu_{f,\phi}}(f) $.
 \end{itemize}
\end{maincorollary}

\begin{corollary}\label{exp}
If $\mathcal{D}^{r} \ni f \mapsto g_{f} \in C^{r}(M , \R)$ is differentiable in $f_{0}$, then the function $\mathcal{D}^{r} \ni f \mapsto \int g_{f} d\mu_{f}$ is differentiable in
$f_{0}$. In particular, if $r \geq 3$ then the following are $C^{r-2}$
 \begin{itemize}
 \item[i.] $\mathcal{D}^{r} \times C^{r}(M , \R) \ni (f , \phi)\mapsto \int \log ||Df(x)|| \, d\mu_{f,\phi} \; ;$
 \item[ii.] $\mathcal{D}^{r} \times C^{r}(M , \R) \ni (f , \phi) \mapsto \int \log ||Df(x)^{-1}||^{-1} \, d\mu_{f,\phi} \; ;$
\item[iii.] $\mathcal{D}^{r} \times C^{r}(M , \R) \ni (f , \phi) \mapsto \int \log |\det Df(x)| \, d\mu_{f,\phi} \; .$
\end{itemize}
\end{corollary}

\subsection{Stability of the statistical laws}

It is classically  known that for expanding dynamics  the Central Limit Theorem holds, and may be obtained in a very elegant way through Cone techniques (see \cite{Vi97}, chapter 2, and \cite{Li95}, Corollary 2).

We also obtain results of stability involving the rates contained in the Central Limit Theorem:

\begin{maintheorem}\label{thm:stalaw1}

Let $r \geq 2$ and $(f, \phi) \in \mathcal{D}^{r} \times C^{\alpha}(M , \R)$. If $\psi \in C^{r}(M , \R)$ then either:
\begin{enumerate}
\item[i.] $\psi = u\circ f - u+\int \psi \,d\mu_{f,\phi}$, for some $u \in C^{r}(M , \R)$
\item[ii.] or the sequence of measurable functions $\frac1{\sqrt{n}} \sum_{j = 0}^{n - 1} \psi \circ f^{j}$
    converge in distribution to the normal distribution of average $m=m_{f,\phi}(\psi)=\int\psi \;d\mu_{f,\phi}$ and variance $\sigma^2$ given by
    $$
    \sigma^2=\sigma_{f,\phi}^{2}(\psi)
        = \int \tilde \psi^{2}\;d\mu_{f,\phi} + 2\sum_{n = 1}^{\infty}\int \tilde \psi(\tilde \psi\circ f^{n}) \;d\mu_{f,\phi}>0,
    $$
    where $\tilde \psi=\psi-\int\psi \;d\mu_{f,\phi}$ is a function with zero average  depending of $f$ and $\phi$. In other words, given a interval $A \subset \R$:
    $$
\mu_{f,\phi}\big(\{ x \in M :\frac{1}{\sqrt{n}}\sum_{j = 0}^{n - 1}\psi(f^{j}(x)) \in A \} \big)\xrightarrow[n \to +\infty]{}
\frac{1}{\sqrt{2\pi}\sigma}\int_{A}e^{\frac{-t^{2}}{2\sigma^{2}}}dt.
$$
\end{enumerate}

Moreover, $\mathcal{D}^{r}\times C^{r}(M , \R) \times C^{r}(M , \R) \ni(f,\phi,\psi) \mapsto m_{f,\phi}(\psi)\;$ and
$\;\mathcal{D}^{r}\times C^{r}(M , \R) \times C^{r}(M , \R) \ni (f,\phi,\psi) \mapsto \sigma^2_{f}(\psi)$ are $C^{r-1}$.

\end{maintheorem}

 Even though the first part of the  statement 
of  Theorem B has some overlap with the classical Central Limit Theorem in the context of expanding maps, the $C^r$ regularity  of $u$ in item i above, is a more recent achievement  for non invertible maps, whose exact  literature reference was not very clear for us. So, we give an independent proof for it (see Proposition \ref{propcohomo} in this paper). On the other hand, the final part of the statement about the $C^r$ dependence of the average and variance of an observable with respect to the dynamics and  the potential is entirely novel.

In particular, we obtain consequences for the study of cohomological equation. Recall that $\psi : M \rightarrow \R$ is cohomologous to a constant for a dynamic
$f : M \rightarrow M$ if there exists a constant $c$ and a continuous function $u : M \rightarrow \R$ such that the cohomological equation $\psi = u\circ f - u + c$ holds.
If $u \in L^{2}(\mu)$ for some probability $\mu$ and the previous equality holds in $\mu-$a.e.p., we say that
$\psi$ is cohomologous to a constant in $L^{2}(\mu)$. When $u \in C^{r}(M, \R)$ and the previous equality holds,  we say that $\psi$ is cohomologous to a constant in $C^{r}(M , \R)$.

Cohomological equations appear in
several problems as smoothness of invariant measures
and conjugacies, mixing properties of suspended flows, rigidity of group actions, and
geometric rigidity questions such as the isospectral problem.
For some other important applications of cohomological equation see for instance \cite{AK11,W13}.

\begin{maincorollary}\label{coho1}
Let $r \geq 2$. If $\psi \in C^{r}(M,\R)$  is not cohomologous to a constant for $f \in \mathcal{D}^{r}$
then the same property holds for all $\hat{f}$ close enough to $f$. As consequence, the sets
$$
\{
    \hat{f}\in \mathcal{D}^{r} : \psi \text{ is cohomologous to a constant } \text{ for } \hat{f}
\} \;\;\text{ and }
$$
$$
\{
    \hat{\psi}\in C^{r}(M,\R) :  \hat{\psi} \text{ is cohomologous to a constant } \text{ for } f
\}
$$
are closed.
\end{maincorollary}

\begin{remark}
The classical statistical stability is sufficient to obtain Corollary \ref{coho1}. Of course, the linear response formula result, which is a stronger result,  also implies it. 
\end{remark}

\subsection{Large deviations stability}

It is well known that expanding dynamics satisfies a large deviations principle for Birkhoff's sums. In other words, given a continuous $\psi : M \rightarrow \R$  the rates
$$
\limsup_{n\to +\infty}\frac{1}{n}\log\mu_{f,\phi}\big( \{x \in M :  |\frac{1}{n}\sum_{i=0}^{n-1}\psi(f^{i}(x)) - \int \psi d\mu_{f,\phi}| \geq \vep \} \big)
\quad \text{and}
$$
$$
 \liminf_{n\to +\infty}\frac{1}{n}\log\mu_{f,\phi}\big( \{x \in M :  |\frac{1}{n}\sum_{i=0}^{n-1}\psi(f^{i}(x)) - \int \psi d\mu_{f,\phi}| > \vep \} \big)
$$ are well understood as functions of the error $\vep > 0$ (for details see e.g. \cite{DK01}).

It is known that for expanding dynamics  the free energy function is well defined. More precisely, given $f \in \mathcal{D}^{r}, \phi ,\psi \in C^{r}(M, \R)$ and $t \in \R$, the expression
$$
\cE_{f,\phi,\psi}(t)
    :=\lim_{n\to\infty} \frac1n \log \int e^{t S_n \psi} \; d\mu_{f,\phi}
   = P_{\mbox{top}}(f, \phi +t\psi) -P_{\mbox{top}}(f, \phi)
$$
 is well defined.
Using that if $\psi$ is not cohomologous to a constant, than the free energy function $\R \ni t\to \cE_{f,\phi,\psi}(t)$ is strictly convex, we can define its Legendre transform
$I_{f,\phi,\psi}$ by
\begin{equation*}
I_{f,\phi,\psi}(s)
    := \sup_{t \in \R} \; \big\{ st-\cE_{f,\phi,\psi}(t) \big\}.
\end{equation*}
(see subsection \ref{subseclarge} for details.)

We obtain the  following stability results about the rates involving large deviations.

\begin{maintheorem}\label{thm:differentiability.LDP1}
Suppose $r \geq 2$. Let $V$ be a compact manifold and $\big((f_v,\phi_{v},\psi_v)\big)_{v\in V}$ an injective and parameterized ($C^{r-1}$) family of maps in
$D^{r} \times C^{r}(M ,\R) \times C^{r}(M ,\R)$.
If the observable $\psi_{v_{\ast}}$ is not cohomologous to a constant,
for some $v_{\ast} \in V$, then there exists an
open neighborhood $U$ of $v_{\ast}$ and an open interval $J$ such that for all $v\in \overline{U}$ and $[a,b]\subset J$
$$
\limsup_{n\to\infty} \frac1n \log \mu_{f_v,\phi_v}
    \left(x\in M : \frac1n \sum_{j=0}^{n-1}\psi_{v}\circ f_{v}^{j}(x) \in [a,b] \right)
    \le-\inf_{s\in[a,b]} I_{f_v,\phi_v,\psi_v}(s)
\;\;\;\text{ and }
$$
$$
\liminf_{n\to\infty} \frac1n \log \mu_{f_v,\phi_v}
    \left(x\in M : \frac1n \sum_{j=0}^{n-1}\psi_v\circ f_{v}^{j}(x) \in (a,b) \right)
    \ge-\inf_{s\in(a,b)} I_{f_v,\phi_v,\psi_v}(s).
$$
Moreover, the Legendre transform  $J \times U \ni (s,v) \mapsto I_{f_v,\phi_v,\psi_v}(s)$ is $C^{r-1}$.

\end{maintheorem}

\subsection{Multifractal analysis}\label{sec:multanal}

Given a continuous observable $\psi: M \rightarrow \R$, a \emph{Multifractal analysis} of its Birkhoff's average means
the study from the topological, dimensional or ergodic viewpoint
of level sets $\{x \in M : \limsup\frac{1}{n}\sum_{i=0}^{n-1}\psi(f^{i}(x)) \in J\}$, where $J$ is an interval. In \cite{BV15} these level sets are studied from the topological pressure viewpoint for dynamics exhibiting an exponential large deviations principle and  a unique equilibrium state which is also a weak Gibbs measure. In particular, \cite[Theorem B]{BV15} states that if $(f , \phi) \in \mathcal{D}^{r} \times C^{r}(M , \R)$ and $\phi \in C^{r}(M , \R)$ is an observable which is not cohomologous to a constant, then
 $$
P_{\overline{X}_{\mu_{f,\phi},\psi,c}}(f,\phi) = P_{\topp}(f , \phi) -\inf_{|s - \int \psi d\mu_{f,\phi}| \geq c} I_{f,\phi,\psi}(s),
$$
where $P_{\overline{X}_{\mu_{f,\phi},\psi,c}}(f,\phi)$ is the  topological pressure of the set
$$
\overline{X}_{\mu_{f,\phi},\psi,c} := \{x \in M : |\limsup\frac{1}{n}\sum_{i=0}^{n-1}\psi(f^{i}(x)) - \int \psi d\mu_{f,\phi}| \geq c \}
$$
 for $f$ with respect the $\phi$.

Recall that a system $f$ satisfies the \emph{specification property} if
for any $\vep>0$ there exists an integer $N=N(\vep)\geq 1$ such that
the following holds: for every $k\geq 1$, any points $x_1,\dots,
x_k$, and any sequence of positive integers $n_1, \dots, n_k$ and
$p_1, \dots, p_k$ with $p_i \geq N(\vep)$
there exists a point $x$ in $M$ such that
$$
\begin{array}{cc}
d\Big(f^j(x),f^j(x_1)\Big) \leq \vep, &\forall \,0\leq j \leq n_1
\end{array}
$$
and
$$
\begin{array}{cc}
d\Big(f^{j+n_1+p_1+\dots +n_{i-1}+p_{i-1}}(x) \;,\; f^j(x_i)\Big)
        \leq \vep &
\end{array}
$$
for every $2\leq i\leq k$ and $0\leq j\leq n_i$. It is a classical result that
expanding maps satisfy the specification property.

Given an expanding map $f$, let us write $\mathcal{M}_{1}(f)$  for the space of $f-$invariant Borel probabilities.
As a consequence of Theorem~\ref{thm:differentiability.LDP1}
we obtain:

\begin{maincorollary}\label{corexp}

Suppose $r \geq 2$. Let $V$ be a compact manifold and $\big((f_v,\phi_{v},\psi_v)\big)_{v\in V}$ an injective and parameterized ($C^{r-1}$) family of maps in
$D^{r} \times C^{r}(M ,\R) \times C^{r}(M ,\R)$. Then, the set
  $Y:= \{(v,c) \in V \times \R^+_0 : c < \sup_{\eta \in \mathcal{M}_{1}(f_{v})}|\int \psi_{v} d\mu_{f_{v},\phi_{v}} - \int \psi_{v} d\eta|\}$ is an open set and
  $$
  Y \ni (v, c) \mapsto P_{\overline{X}_{\mu_{f_{v},\phi_{v}},\psi_{v},c}}(f_{v},\phi_{v});
  $$
is $C^{r-1}$.

\end{maincorollary}

Define the spectrum of a continuous function
$\psi : M \rightarrow \R$ by:
$$
L(\psi, f) := \{ \alpha \in \R : \,\text{there exists} \,\, x \in M \,\, \text{such that} \,\,
lim_{n \to +\infty}\frac{1}{n}\sum_{i = 0}^{n-1}\psi (f(x)) = \alpha\}.
$$

  The set $Y$ defined above  has a close relationship with the Birkhoff's spectrum of the observable $\psi$. As $f \in \mathcal{D}^{r}$ has specification property we know that
$$
L(\psi,f) = [\inf_{\eta \in \mathcal{M}_{1}(f)}\int \psi d\eta \;,\; \sup_{\eta \in \mathcal{M}_{1}(f)} \int \psi d\eta]
$$
(see \cite[Lemma 1.1]{Tho9}).
Thereby, $Y$ is the set of
$(v,c) \in V \times \R^+_0$ such that $\int \psi_{v}d\mu_{f_{v},\phi_{v}} + c$ or $\int \psi_{v}d\mu_{f_{v},\phi_{v}} - c$ belongs to the interior of $L(\psi_{v},f_{v})$.


\begin{remark}
As we shall see, the results of high regularity are consequences of the uniform spectral gap of the transfer operator on smooth functions space. Thus, similar results can be obtained for a robust non-uniformly expanding setting and potential close to a constant function studied
in \cite{BCV16}.
\end{remark}

\section{Preliminary}\label{prelim}

\subsection{Cones and projective metrics}

In this section we recall basic elements of the cones and projective metrics theory, for more details see e.g. \cite{Vi97,AM06}.

Let $E$ be a vector space, $\emptyset \neq C \subset E\setminus\{0\}$ is a (convex) cone  if $\forall v_{1}, v_{2} \in C$ and $t > 0$ we have $tv_{1} + v_{2} \in C$. Requiring that $C \cap (-C) = \emptyset$ we can induce a partial order on $E$ that preserves its structure of vector space; in fact:
$$u \preceq v \Leftrightarrow v - u \in C\cup\{0\}$$

The closure $\overline{C}$ of $C$ is defined by:
$$w \in \overline{C} \Leftrightarrow \text{ there exists }  v \in C \text{ and } t_{n}\searrow 0 \text{ such that } (w + t_{n}v) \in C,\, \forall n \in \N.$$

We will work from here with cones such that $\overline{C} \cap (-\overline{C}) = \{0\}$,
this will allow us to define a pseudo-metric on the cone. Given $v_{1}$ and $v_{2} \in C$, define:

\begin{itemize}
  \item $\alpha(v_{1} , v_{2}) := \sup\{t > 0; v_{2} - tv_{1} \in C \}$
  \item $\beta(v_{1} , v_{2}) := \inf\{t > 0; tv_{1} - v_{2} \in C \}$.\\
\end{itemize}

Let $\Theta : C \times C \rightarrow [0 , +\infty]$ be defined by:
$$\Theta(v_{1} , v_{2}) := \log\frac{\beta(v_{1} , v_{2})}{\alpha(v_{1} , v_{2})}.$$
 $\Theta$ is known as Hilbert's metric, and it is a pseudo-metric.

It is natural to wonder about the relationship between the projective metric and the pre-existing metrics on a normed space, in general depends of the cone that we have defined. The next proposition (Lemma 8.2 in \cite{AM06}) gives a result in this direction.

\begin{proposition}\label{cmpn2}
Let $E$ be a normed space, $|| \cdot ||_{i}$ be semi-norms on $E$, for $i = 1 , 2$, and $\preceq$ be a partial order that preserve its structure of vector space. Suppose that for all $v, u \in C$ holds:
$$-v \preceq u \preceq v \Rightarrow ||u||_{i} \leq ||v||_{i}, i = 1, 2.$$
Then given $f , g \in C$, with $||f||_{1} = ||g||_{1}$, we have:
$$||f - g||_{2} \leq (e^{\Theta(f , g)} - 1)||f||_{2}.$$
\end{proposition}

A fundamental result in cones theory assures that if the image of a cone $C_1$ by a linear map $L$ is bounded with respect to the projective metric of some cone $C_2$ that contains $L(C_1)$, then the map $L$ is a  contraction with respect to such projective metrics:

\begin{theorem}\label{contracao}
Let $E_{1} , E_{2}$ vector spaces, $C_{i} \subset E_{i}, i = 1 , 2,$ cones, $L : E_{1} \rightarrow E_{2}$ a linear operator such that $L(C_{1}) \subset C_{2}$ and $D := \sup\{\Theta_{2}(L(u) , L(v)) : u ,v \in C_{1} \}$. If $D < +\infty$ then:
$$\Theta_{2}(L(u) ,L(v)) \leq (1 - e^{-D})\Theta_{1}(u , v), \forall\,u ,v \in C_{1}.$$
\end{theorem}


\subsection{Uniform spectral gap}

In this section we are interested in proving that given $(f, \phi) \in \mathcal{D}^{r}\times C^{r}(M , \R)$ the transfer operator $\mathcal{L}_{f,\phi}$ has spectral gap property as a operator acting on  $C^{r}(M , \C)$. Furthermore, such spectral gap can be taken  uniform under small perturbations in dynamics and potentials.

Usually, spectral gap for the transfer operator means that the spectrum of the operator can be decomposed into two compact disjoint components.
The elements of one of this components are eingenvalues whose eigenspaces are finite dimensional.  The other component has its norm
strictly smaller than the norm of the eingenvalues in the first component. This kind of decomposition often arises from Lasota-Yorke arguments.  

In this paper, what we mean by the {\em spectral gap property}, is
stronger than the usual definition. We require that the first component 
has exactly one dominant eigenvalue whose eigenspace is one-dimensional:

\begin{definition}{({\em Spectral gap property.})}\label{desgp}
 Given a Banach $E$ and a bounded linear operator $A : E \rightarrow E$, we will say that $A$ has the \emph{spectral gap property} if its spectrum $spec(A)$ can be decomposed into two (compact) spectral components $\Sigma$ and $\{\lambda\}$ such that $\lambda$ is the spectral radius of $A$, is a simple positive eigenvalue
 and $\Sigma \cap B(0, \lambda) = \Sigma$.
\end{definition}
Note that  the {\em spectral gap property} for an operator $A$  is equivalent to the existence  of a splitting $E = E_{0} \oplus E_{1}$ such that:
\begin{itemize}
 \item $E_{0}$ and $E_{1}$ are closed $A-$invariant subspaces
;
\item
$E_{1}$ is a one-dimensional eigenspace
associated to the spectral radius $\lambda$, and
\item There exists $\tau \in (0 , 1)$ and $k \geq 0$ such that for all $\varphi \in E_{0}$ we have
$||\frac{A^{n}\varphi}{\lambda^{n}}|| \leq ||\varphi|| k\tau^{n},$ for every $n \geq 0$.
\end{itemize}

To show that $\mathcal{L}_{f,\phi|C^{r}}$ has the spectral gap property in the case of smooth expanding dynamics we will use the technique of cones. In other words, we find a
convex invariant cone by $\mathcal{L}_{f,\phi}$ within the space of strictly positive functions whose image by $\mathcal{L}_{f,\phi}$ has
finite diameter in the projective metric associated to cone and so we get convergence in the  $C^{r}-$norm  through the convergence in the projective metrics. For examples of use of the cones technique see  \cite{Li95,Vi97,LSV98,Ba00, Cas02, Cas04, AM06, BS09,CV13, BCV16, CN17}.

For $r \geq 1$ and $\kappa > 0$ we introduce the cone of functions
$$
\Lambda_{\kappa}^{r} := \{ \varphi \in C^{r}(M , \R) : \varphi > 0 \; \mbox{and} \; \sup_{x \in M}\Big|\Big|\frac{D^{s}\varphi(x)}{\varphi(x)}\Big|\Big| \leq \kappa c_{r ,s}, \forall 1 \leq s \leq r \},
$$
where $c_{r , r} = 1$. For $r \neq s$,  $c_{r , s}$ are constants  chosen below in order to provide the cone invariance:
 Set 
\[
\begin{cases}
c_{r, r}=1;  \\
c_{r, r-1}= \frac{(1-\sigma^{-1})}{2^{r+2} (r+ 1)!  e^{\|\phi\|_r} \max_{1\le k\le r-1}\max_{1\leq j\leq deg(f)}\{1, \sup_{z \in M} \|D^{r}f_j(z)\|^k\}} ; \\
c_{r, r-t} = c_{r, r-1} \cdot c_{r-1, r-t}, \text{ for } t=2\dots r-1
\end{cases}
\]
Roughly, the choice of $c_{r, s}$ is made in order to guarantee that the at most $2^{r+1} (r+1)!$ terms arising in the
computation of  higher order derivatives of the observable $\cL_\phi\varphi$ are dominated by the term involving
$D^r\varphi$,  while the recursive choice of the constants $c_{r, s}$ with $s<r$ guarantees that the cones 
corresponding to lower order differentiability are contracted.  Hence, our main result in this section is as follows.

For simplicity, we omit the dependence of $c_{r, s} $ in the notation of the cones.

When $r = 1$ we have
$$
\Lambda_{\kappa}^{1} = \Big\{ \varphi \in \mathcal{C}^{1}(M , \R) : \varphi > 0 \;\;\;\mbox{ and }\;\;\; \sup_{x \in M}\Big|\Big|\frac{D\varphi(x)}{\varphi(x)}\Big|\Big| \leq \kappa \Big\}.
$$

We have that $\Lambda_{\kappa}^{r}$ is a convex cone; furthermore
$\overline{\Lambda_{\kappa}^{r}} \cap (-\overline{\Lambda_{\kappa}^{r}}) = \{0\}$.

Now we are able to prove the cone invariance property.

\begin{lemma} Suppose that for some  $0< \rho< 1$ there exists $\kappa_{0}$ such that  $\mathcal{L}_{f, \phi} \Lambda_{\kappa}^{r-1} \subset \Lambda_{ \rho\kappa}^{r-1} $ for all $\kappa \geq \kappa_0$. Suppose also that 
for some $\kappa_r$ 
$$
\sup_{x \in M}\Big\|\frac{D^r\mathcal{L}_{f , \phi}\varphi(x)}{\mathcal{L}_{f, \phi}\varphi(x)}\Big\| \leq \rho \kappa, \forall
\kappa\geq \kappa_r.
$$
Then, there exists some $\kappa_1$ such that 
$\mathcal{L}_{f, \phi} (\Lambda_{\kappa}^{r} )\subset \Lambda_{ \rho\kappa}^{r} $ for all $\kappa \geq \kappa_1$.
\end{lemma}
\begin{proof}

Note that $\mathcal{L}_{f, \phi}(\Lambda_{\kappa}^{r-1}) \subset \Lambda_{ \rho\kappa}^{r-1}, \, \forall  \kappa \geq \kappa_0$ implies  in particular that 
$$
\mathcal{L}_{f, \phi}(\Lambda_{c_{r, r-1} \kappa}^{r-1}) \subset \Lambda_{\rho c_{r, r-1}\kappa}^{r-1} \forall  \kappa \geq c_{r, r-1}^{-1}\kappa_0.
$$
Take $\kappa_1= \max\{ c_{r, r-1}^{-1}\kappa_0, \kappa_r\}$
$$
\Big\|{\frac{D^j\mathcal{L}_{f , \phi}\varphi(x)}{\mathcal{L}_{f, \phi}\varphi(x)}}\Big\| \leq \rho \kappa, \forall \kappa \geq \kappa_1.
$$

Note that $\Lambda_{\rho\kappa}^{r}= \Lambda_{c_{r, r-1}\rho\kappa}^{r-1} \cap \{\vr \in C^r; \Big\|{\frac{D^r\mathcal{L}_{f , \phi}\varphi(x)}{\mathcal{L}_{f, \phi}\varphi(x)}}\Big\| \leq \rho \kappa \}$ and this implies the lemma.

\end{proof}

\begin{proposition}\label{inv}
Given $f \in D^{r}$ and $\phi \in C^{r}$ there exists $\kappa_{0} > 0, c_{r , s} > 0$ and $\rho \in (0 , 1)$ such that $\mathcal{L}_{f , \phi} \Lambda_{\kappa}^{r} \subset \Lambda_{\rho\kappa}^{r},$ for all $\kappa \geq \kappa_{0}$.
\end{proposition}

\begin{proof}

Let $\varphi \in \Lambda_{\kappa}^{r}$. Then 
$\mathcal{L}_{f, \phi}\varphi > 0$. 
Given $x \in M$, let us write $f_j$, $j= 1, \dots, \deg(f)$ for the local inverse branches of $f$. 
For $H \in T_{x}M$ with $||H|| = 1$ we have that:
\begin{align*}
\frac{|D\mathcal{L}_{f , \phi}\varphi(x) \cdot H|}{\mathcal{L}_{f , \phi}\varphi(x)} & = \frac{|\sum_{j = 1}^{\deg(f)}\big[ e^{\phi(f_{j}x)}\varphi(f_{j}x)D\phi(f_{j}x) \cdot Df_{j}(x) \cdot H + e^{\phi(f_{j}x)}D\varphi(f_{j}x) \cdot Df_{j}(x) \cdot H \big]|}{\sum_{j = 1}^{\deg(f)} e^{\phi(f_{j}x)}\varphi(f_{j}x)} \\
& \leq \frac{\sum_{j = 1}^{\deg(f)}\big| e^{\phi(f_{j}x)}\varphi(f_{j}x)D\phi(f_{j}x) \cdot Df_{j}(x) \cdot H\big| + \big| e^{\phi(f_{j}x)}D\varphi(f_{j}x) \cdot Df_{j}(x) \cdot H \big|}{\sum_{j = 1}^{\deg(f)} e^{\phi(f_{j}x)}\varphi(f_{j}x)} \\
& \leq ||D\phi||_{0}\sigma^{-1} + \sigma^{-1}\kappa,
\end{align*}
hence $\sup_{x \in M}||\frac{D\mathcal{L}_{f , \phi}\varphi(x)}{\mathcal{L}_{f, \phi}\varphi(x)}|| \leq ||D\phi||_{0}\sigma^{-1} + \sigma^{-1}\kappa.$ Taking $\kappa_{0} := \frac{2||D\phi||_{0}}{\sigma- 1}$ we have that
$$
\sup_{x \in M}\Big\|\frac{D\mathcal{L}_{f , \phi}\varphi(x)}{\mathcal{L}_{f, \phi}\varphi(x)}\Big\|  \leq \frac{1 - \sigma^{-1}}{2}\kappa,
$$
for all $\kappa \geq \kappa_{0}.$ Thus we prove the proposition for the case $r = 1.$

In order to show how the argument works for higher orders,  let us first 
consider  the case $r = 2$. Take $\kappa > 0$ and $\varphi \in \Lambda^{2}_{\kappa}$. Using the chain rule, we have that $D^{2}(\mathcal{L}_{f, \phi} \varphi)(x)$  is a sum of the following seven terms:
\begin{eqnarray}
   & D^{2}\phi(x_{j})[Df_{j}(x)]^{2}e^{\phi(x_{j})}\varphi(x_{j})& \nonumber\\
   & D\phi(x_{j})D^{2}f_{j}(x)e^{\phi(x_{j})}\varphi(x_{j})& \nonumber\\
   & D\varphi(x_{j})Df_{j}(x)e^{\phi(x_{j})}D\phi(x_{j})Df_{j}(x) &\nonumber\\
   & D\varphi(x_{j})Df_{j}(x)e^{\phi(x_{j})}D\phi(x_{j})Df_{j}(x) &\nonumber\\
   &  \varphi(x_{j})D\phi(x_{j})Df_{j}(x)e^{\phi(x_{j})}D\phi(x_{j})Df_{j}(x) &\nonumber\\
   &  e^{\phi(x_{j})}D^{2}\varphi(x_{j})[Df_{j}(x)]^{2} & \nonumber\\
   &  e^{\phi(x_{j})}D\varphi(x_{j})D^{2}f_{j}(x). & \nonumber
\end{eqnarray}

Hence, for $x \in M$ and $H \in T_{x}M$, with $||H|| = 1$, we have:
$$
\frac{|D^{2}\mathcal{L}_{f , \phi}\varphi(x) \cdot H|}{\mathcal{L}_{f , \phi}\varphi(x)}  \leq  ||\phi||_{2}\sigma^{-1} +
||\phi||_{2} \cdot \sup_{z \in M}\{\|[Df(z)]^{-1}\|\} \cdot ||f||_{2} \cdot \sigma^{-1} +   2c^{1}_{2, 1}\kappa \sigma^{-1}||\phi||_{2} +
$$
$$
||\phi||_{2}^{2}\sigma^{-1} +
\kappa\sigma^{-1} + c^{1}_{2,1}\kappa \cdot \sup_{z \in M}\{\|[Df(z)]^{-1}\|\} \cdot ||f||_{2} \cdot \sigma^{-1}
=
$$
$$
||\phi||_{2}\sigma^{-1}\big( 1 + \sup_{z \in M}\{\|[Df(z)]^{-1}\|\} \cdot ||f||_{2}  + ||\phi||_{2} \big) +
\sigma^{-1}\kappa + \sigma^{-1}c^{1}_{2,1}\kappa\big(\sup_{z \in M}\{\|[Df(z)]^{-1}\}\| \cdot ||f||_{2}  + 2||\phi||_{2}\big).
$$
So
$$
\sup_{x \in M}\Big|\Big|\frac{D^{2}\mathcal{L}_{f , \phi}\varphi(x)}{\mathcal{L}_{f, \phi}\varphi(x)}\Big|\Big| \leq ||\phi||_{2}\sigma^{-1}\big( 1 + \sup_{z \in M}\{\|[Df(z)]^{-1}\|\} \cdot ||f||_{2}  + ||\phi||_{2} \big) +
\sigma^{-1}\kappa +
$$
$$
\sigma^{-1}c^{1}_{2,1}\kappa\big(\sup_{z \in M}\{\|[Df(z)]^{-1}\|\} \cdot ||f||_{2}  + 2||\phi||_{2}\big).
$$
Taking $\kappa_{0} := \frac{3 ||\phi||_{2}\big( 1 + \sup_{z \in M}||[Df(z)]^{-1}|| \cdot ||f||_{2}  + ||\phi||_{2} \big)}{\sigma - 1}$ and
$c_{2,1} := \frac{\sigma - 1}{3\big(\sup_{z \in M}||[Df(z)]^{-1}|| \cdot ||f||_{2}  + 2||\phi||_{2}\big)}$ we have
$$
\sup_{x \in M}\Big|\Big|\frac{D^{2}\mathcal{L}_{f , \phi}\varphi(x)}{\mathcal{L}_{f, \phi}\varphi(x)}\Big|\Big| \leq \frac{2 + \sigma^{-1}}{3}\kappa,
$$
for all $\kappa > \kappa_{0}.$ Therefore, using the case $r = 1$ and previous lemma, we prove the proposition for the case $r = 2.$

Now we proceed to the general case. It follows through an analogous computation of higher order derivatives of $\mathcal{L}_{f,\phi}\varphi$, by using the chain rule. Also by the previous lemma, all we need to prove is that 
$$
\sup_{x \in M}\Big|\Big|\frac{D^{r}\mathcal{L}_{f , \phi}\varphi(x)}{\mathcal{L}_{f, \phi}\varphi(x)}\Big|\Big| \leq \rho \kappa 
$$
for some $0<\rho < 1$.

Suppose 
that  the expression of the $k-$th derivative of $\mathcal{L}_{f,\phi}\varphi$ is a sum with  $m_k$  terms. Suppose that each one of this terms is a product with, at most, $p_k$ factors. 
As we apply the  chain rule, we obtain that 
 $m_{k+1}\leq m_k \cdot p_k$ and $p_{k+1}= p_{k}+ 2$.
 Since $n_1= 2$ and $p_1= 4$, by induction, this implies that 
 $m_r \leq 2^{r+1} (r+ 1)!$. Analogously to the case $r= 2$ 
 that we have already proved,  after we apply the triangle inequality, we can group the terms in the sum of the expression of $\frac{D^{r}(\mathcal{L}_{f, \phi} \varphi)}{\mathcal{L}_{f, \phi} \varphi}$ into three classes:
 \begin{enumerate}
 \item
 A class containing the unique term that involves $D^{r}\varphi$, whose norm has the following upper bound: 
 $$
\frac{ \sum_j \|e^{\phi(x_{j})}\|\|D^{r}\varphi(x_{j})\|\|[Df_{j}(x)] \|^r}{\sum_j e^{\phi(x_{j})} \varphi(x_j)}\leq \kappa \sigma^{-r}
 $$
 
 \item
 The class of terms involving $D^{k}\varphi$, for $1 \leq k < r$. 
 Given a term in this class,  its norm is bounded by:
 $$
\frac{ \sum_j \|e^{\phi(x_{j})}\|\|D^{b} \phi(x_j)\|\|D^{k}\varphi(x_{j})\|\cdot \Pi_{q; \sum i_q= r}\|[D^{i_q}f_j(x)]\|}{\sum_j e^{\phi(x_{j})} \varphi(x_j)}\leq  \qquad \qquad \qquad \qquad \qquad 
$$
$$
 \qquad \quad \kappa c_{r, k} \cdot e^{\|\phi\|_r} \max_{j, 1\le k\le r-1}\{1, \sup_{z \in M}\{ \|D^{k}f_j(z)\|^r\} \}\leq\kappa  \frac{1- \sigma^{-1}}{2^{r+2}(r+1)! },
 $$
 where $1\leq b \leq r- k$ can vary according to the term  in the class.
Recall that this class is a sum with at most $2^{r+ 1}(r+1)!$ terms.

 Finally, we have:
 
 \item 
 The class whose unique term involves $D^{r}\phi$. Such term has its norm bounded by:
  $$
\frac{\sum_j  \|e^{\phi(x_{j})}\|\|D^{r} \phi(x_j)\|\|\varphi(x_{j})\|\|[Df_{j}(x)]\|^{r}}{\sum_j e^{\phi(x_{j})} \varphi(x_j)}\leq \sup_{z\in M} \{\|Df_j(z)\|^r\} \| \phi\|_r := \hat \kappa.
 $$
 
 \end{enumerate}

Set $\kappa_r:= \frac{4 \hat \kappa}{1- \sigma^{-1}}$.
From the estimatives above, for any $\kappa\geq \kappa_r$ we obtain that 
$$
\sup_{x \in M}\Big|\Big|\frac{D^{r}\mathcal{L}_{f , \phi}\varphi(x)}{\mathcal{L}_{f, \phi}\varphi(x)}\Big|\Big| \leq \kappa \sigma^{-r}+ 
\kappa \frac{1- \sigma^{-1}}{2} + \hat \kappa \leq \kappa \frac{3+  \sigma^{-1}}{4}.
$$
So, just take $\rho= \frac{3+ \sigma^{-1}}{4}$ in the previous lemma.

\end{proof}

The next corollary tells us that  we can take  a constant invariant cone on  sufficiently small neighborhoods of
 the dynamic and
potential. It will be fundamental in the uniformity of the spectral gap.

\begin{corollary}\label{unigap}
Given $f \in D^{r}$ and $\phi \in C^{r}(M , \R)$ there exists neighborhoods $\mathcal{F}^{r}$ of $f$ and $\mathcal{W}^{r}$ of $\phi$, as well as constants
$\kappa > 0, c_{r , s} > 0$ and $\rho \in (0 , 1)$ such that:  if $(\hat{f} , \hat{\phi}) \in \mathcal{F}^{r} \times\mathcal{W}^{r}$ then $\mathcal{L}_{\hat{f} , \hat{\phi}} \Lambda_{\kappa}^{r} \subset \Lambda_{\rho\kappa}^{r}$.
\end{corollary}
\begin{proof}
It follows directly from the estimates  in the previous proposition, which are uniform for sufficiently small neighborhoods of $f$ and $\phi$.
\end{proof}

The next proposition, among other things, implies that the image of the cone $\Lambda^{r}_{\kappa}$  by the transfer operator  has finite diameter, 
for any sufficiently large values of $\kappa$.

\begin{proposition}
Given $0 < \rho < 1$, the cone $\Lambda^{r}_{\rho\kappa}$ has finite diameter with respect to  the projective metric induced by cone $\Lambda^{r}_{\kappa}$.
\end{proposition}

\begin{proof}
Let us write $\theta_{\kappa}$ for the projective metrics of  $\Lambda^{r}_{\kappa}$.
Due to Theorem \ref{contracao} all we need to prove is that $\theta_{\kappa}(\varphi , 1)$ is uniformly bounded for all $\varphi \in \Lambda^{r}_{\rho\kappa}$. Take $\varphi \in \Lambda^{r}_{\rho\kappa}$.\\

\emph{Claim 1: $\beta_{\kappa}(\varphi , 1) \leq \frac{1}{\inf
\varphi(1 - \rho)}$.}\\

\noindent Set $t_{0} := \frac{1}{\inf \varphi (1 - \rho )}$. We will prove that $t_{0}\varphi - 1 \in \Lambda^{r}_{\kappa}$. 
As $\rho < 1$ we have $t_{0} \varphi - 1 > 0$, 
and given $x \in M, H \in T_{x}M$, with $||H|| = 1$, 
$$
\frac{|D^{s}(t_{0}\varphi - 1)(x) \cdot H|}{t_{0} \varphi(x) - 1} = \frac{t_{0}|D^{s}\varphi(x) \cdot H|}{t_{0}\varphi(x) - 1} \leq \frac{t_{0}\varphi(x)}{t_{0}\varphi(x) - 1} \cdot \rho\kappa c_{r , s}\leq \kappa c_{r , s} .
$$
\\

\emph{Claim 2: $\alpha_{\kappa}(\varphi , 1) \geq \frac{1}{\sup
\varphi(1 + \rho )}$.}\\

\noindent In fact, we will prove that for $t_{1} := \frac{1}{\sup \varphi(1 + \rho )}$ we have  $1 - t_{1}\varphi \in \Lambda^{r}_{\kappa}$. We have that $1 - t_{1} \varphi > 0$ and given $x \in M, H \in T_{x}M$, with $||H|| = 1$, we obtain 
$$
\frac{|D^{s}(1 - t_{1}\varphi)(x) \cdot H|}{1 - t_{1} \varphi(x)} \leq \frac{t_{1}|D^{s}\varphi(x) \cdot H|}{1 - t_{1} \varphi(x) } \leq \frac{t_{1}}{1 - t_{1} \varphi(x)} \cdot \varphi(x)\rho\kappa c_{r , s} \leq \kappa c_{r , s}.
$$
\\

It follows from  \emph{Claim 1 and Claim 2} that
\begin{align*}
\theta(\varphi , 1) & \leq \log\frac{\sup \varphi(1 + \rho )}{\inf \varphi (1 - \rho)} \\
                    & = \log \frac{\sup\varphi}{\inf\varphi} + \log\frac{1 + \rho }{1 - \rho} \\
                    & \leq \sup_{x \in M}\{\frac{||D\varphi(x)||}{\varphi(x)}\} \mbox{diam}(M) + \log\frac{1 + \rho }{1 - \rho} \\
                    & \leq \rho\kappa\mbox{diam}(M) + \log\frac{1 + \rho }{1 - \rho}.
\end{align*}

\end{proof}

Now we will prove the spectral gap property in
$C^{r}(M , \C)$. Before the proof, recall that follows from Geometric Hahn-Banach's Theorem that there exists a probability $\nu_{f ,
\phi}$ such that $\mathcal{L}_{f , \phi}^{\ast}\nu_{f , \phi} =
\lambda_{f , \phi}\nu_{f , \phi}$, where $\lambda_{f , \phi}$ is the spectral radius of $\mathcal{L}_{f, \phi}$ acting on $C^{0}$.
Fix a probability $\nu_{f ,\phi}$  with this property. We shall see that $\nu_{f , \phi}$ is unique.
We denote $\frac{\mathcal{L}_{f ,\phi}}{\lambda_{f , \phi}}$
by $\tilde{\mathcal{L}}_{f ,\phi}$.

\begin{theorem}
If $f \in D^{r}$ and $\phi \in C^{r}$ then $\mathcal{L}_{f,\phi|\mathcal{C}^{r}}$ has the spectral gap property.
\end{theorem}

\begin{proof}

Take $\kappa_{0} , c_{r,s}$ and $\rho$ as in the proposition
\ref{inv}. Let $\varphi, \psi \in \Lambda_{\kappa_{0}}^{r}$ be and
$\theta_{+}$ be the projective metric associated to cone of the positive functions. By theorem \ref{contracao}, for $n , k \geq 1$
we have:
\begin{equation}\label{equa5}
\theta_{+}\big(\tilde{\mathcal{L}}^{n + k}_{f,\phi}(\varphi) ,
\tilde{\mathcal{L}}^{n}_{f,\phi}(\psi)\big) \leq
\theta_{\kappa_{0}}\big(\tilde{\mathcal{L}}^{n +
k}_{f,\phi}(\varphi) , \tilde{\mathcal{L}}^{n}_{\phi}(\psi)\big)
\leq \Delta \tau^{n-1},
\end{equation}
where $\Delta$ is the $\theta_{\kappa_{0}}-$diameter of the cone
$\Lambda_{\rho\kappa_{0}}^{r}$ and $\tau := 1 - e^{-\Delta} \in (0 ,
1)$. Note that $\big(\varphi_{n} :=
\tilde{\mathcal{L}}^{n}_{f,\phi}(\varphi)\big)_{n \geq 1}$ is
Cauchy with respect to  the metrics $\theta_{+}$. It is well known that $\theta_{+}$ is
complete (see e.g. \cite{Vi97}), so there exists $h_{\varphi} \in
C_{+}$ such that $\tilde{\mathcal{L}}^{n}_{f,\phi}(\varphi)
\xrightarrow[]{\theta_{+}}h_{\varphi}$ and $\int h_{\varphi}d\nu_{f
, \phi} = \int \varphi d\nu_{f ,\phi}$. As
$\int\tilde{\mathcal{L}}^{n}_{f,\phi}(\varphi)d\nu_{f,\phi} = \int
\varphi d\nu_{f,\phi}$ we can apply the proposition \ref{cmpn2}.  We just take the second seminorm in the statement
of proposition  \ref{cmpn2} to be 
 the $C_0$-norm (denoted by $\| \cdot \|_0$) and the first  seminorm to be the absolute value of the integral with respect to $\nu_{f,\phi}$. 
 
 Thus
$\tilde{\mathcal{L}}^{n}_{f,\phi}(\varphi)
\xrightarrow[]{\mathcal{C}^{0}}h_{\varphi}$ and thereby
$\mathcal{L}_{f,\phi}h_{f,\varphi} = \lambda_{f ,\phi}
h_{\varphi}$.\\

\emph{Claim 1: If $\varphi \in \Lambda^{r}_{\kappa_{0}}$ then $||\varphi_{n+k} -  \varphi_{n}||_{r} \leq \Delta \tau^{n-1}\big( \kappa_{0}(e^{\Delta} + 2) + e^{\Delta} \big) ||h_{\varphi}||_{0}.$ }\\

\noindent In fact, applying the proposition \ref{cmpn2} as explained in the beginning of this proof, 
besides the estimate (\ref{equa5}), we have
\begin{align}\label{equa9}
||\varphi_{n+k} -  \varphi_{n}||_{0} & \leq  \big(e^{\theta^{+}(\varphi_{n+k} , \varphi_{n})} -1 \big) \cdot ||\varphi_{n+k}||_{0} \nonumber \\
                                     & \leq  \theta^{+}(\varphi_{n+k} , \varphi_{n})e^{\Delta}||\varphi_{n+k}||_{0} \nonumber \\
                                     & \leq  \Delta \tau^{n-1}e^{\Delta}||\varphi_{n+k}||_{0}.
\end{align}
Note also that, as $\int\varphi_{n} d\nu_{f,\phi} = \int\varphi_{n + k}
d\nu_{f,\phi}$ and $\varphi_{n}, \varphi_{n +k}$ are strictly positive continuous functions then
$$
\beta_{\kappa_{0}}(\varphi_{n} , \varphi_{n+k}) \geq 1 \geq \alpha_{\kappa_{0}}(\varphi_{n} , \varphi_{n+k}),
$$
hence
$$
\Theta_{\kappa_{0}}(\varphi_{n} , \varphi_{n+k }) \leq \Delta \tau^{n-1} \Rightarrow \frac{\beta_{\kappa_{0}}(\varphi_{n} , \varphi_{n+k})}{\alpha_{\kappa_{0}}(\varphi_{n} , \varphi_{n+k})} \leq e^{\Delta\tau^{n-1}} \Rightarrow |1 - \alpha_{\kappa_{0}}(\varphi_{n} , \varphi_{n+k})| \leq \Delta\tau^{n-1}.
$$
Thus,
$$
||D^{s}(\varphi_{n} - \varphi_{n+k} )||_{0}  \leq  ||D^{s}(\varphi_{n} - \alpha_{\kappa_{0}}(\varphi_{n} , \varphi_{n+k}) \cdot \varphi_{n+k}) ||_{0} + |\alpha_{\kappa_{0}}(\varphi_{n} , \varphi_{n+k}) - 1| \cdot ||D^{s} \varphi_{n+k} ||_{0} \leq
$$
$$
c_{r , s} \kappa_{0} ||\varphi_{n} - \alpha_{\kappa_{0}}(\varphi_{n} , \varphi_{n+k}) \cdot \varphi_{n+k}||_{0} + |\alpha_{\kappa_{0}}(\varphi_{n} , \varphi_{n+k}) - 1| \cdot ||D^{s} \varphi_{n+k} ||_{0} \leq
$$
$$
c_{r , s} \kappa_{0} \Big( ||\varphi_{n} - \varphi_{n+k}||_{0} + |1 - \alpha_{\kappa_{0}}(\varphi_{n} , \varphi_{n+k})| \cdot ||\varphi_{n+k}||_{0} \Big) + c_{r , s}\kappa_{0}|\alpha_{\kappa_{0}}(\varphi_{n} , \varphi_{n+k}) - 1| \cdot ||\varphi_{n+k} ||_{0} \leq
$$
$$
 c_{r , s}\kappa_{0} \Big[ ||\varphi_{n} - \varphi_{n+k}||_{0} + ||\varphi_{n+k} ||_{0}2\Delta\tau^{n-1} \Big] \leq
$$
$$
c_{r , s}\kappa_{0} \Big[ \Delta \tau^{n-1}e^{\Delta}||\varphi_{n+k}||_{0} + ||\varphi_{n+k} ||_{0}2\Delta\tau^{n-1} \Big] \leq
$$
\begin{equation}\label{equa7}
c_{r , s}\kappa_{0}\Delta \tau^{n-1}||\varphi_{n+k}||_{0}(e^{\Delta} + 2).
\end{equation}
As $\varphi_{n}$ converges uniformly, using the estimates  (\ref{equa9}) and (\ref{equa7}) we have that $\varphi_{n}$ is a Cauchy sequence in $C^{r}$. Furthermore, doing $k \to +\infty$,  we have that
 \begin{equation}\label{equa8}
||h_{\varphi} - \varphi_{n}||_{r} \leq \Delta \tau^{n-1}\big( \kappa_{0}(e^{\Delta} + 2) + e^{\Delta} \big) ||h_{\varphi}||_{0}.
\end{equation}
%
\\

It follows from \emph{Claim 1} that $\varphi_{n}$ converges to $h_{\varphi}$ in the $C^{r}$ norm and $h_{\varphi} \in \Lambda_{\kappa_{0}}^{r}$.\\

\emph{Claim 2: $\ker(\mathcal{L}_{\phi} - \lambda_{f,\phi} I)\cap\mathcal{C}^{r}(M , \C)$ has dimension one.}\\
\noindent In fact; since that $\lambda_{f,\phi} \in \R$ is enough prove that $\ker(\mathcal{L}_{\phi} - \lambda_{f,\phi} I)\cap\mathcal{C}^{r}(M , \R)$ has dimension one. Let $\ds h_{f , \phi} := \lim_{n \to + \infty}\tilde{\mathcal{L}}^{n}_{f,\phi}(1)$  and $u \in \ker(\mathcal{L}_{f,\phi|\mathcal{C}^{r}} -\, \lambda_{f,\phi} I)\,\cap\, \Lambda_{\kappa_{0}}^{r}$. By (\ref{equa5}) there exists $t_{1} > 0$ such that $t_{1}u = h$. Thereby, by \emph{Claim 1}, for all $\varphi \in \Lambda_{\kappa_{0}}^{r}$ we have that $\tilde{\mathcal{L}}^{n}_{f,\phi}(\varphi) \xrightarrow[]{\mathcal{C}^{r}} \int\varphi d\nu_{f, \phi}\cdot h_{f , \phi}$. Given $v \in \ker(\mathcal{L}_{f,\phi|\mathcal{C}^{r}} - \lambda_{f,\phi} I) \cap\mathcal{C}^{r}(M , \R)$, there exists a constant $B > 0$ such that $v + B$ is a element of $\Lambda_{\kappa_{0}}^{r}$; hence $v = \lim\tilde{\mathcal{L}}^{n}_{f,\phi}(v + B) - \lim\tilde{\mathcal{L}}^{n}_{f,\phi}(B) = \int vd\nu_{f,\phi}\cdot h_{f,\phi}$. Thus, $\ker(\mathcal{L}_{f,\phi|\mathcal{C}^{r}}\, -\, \lambda_{f,\phi} I) = \{ th_{f,\phi} : t \in \C \}$. \\

Let $E_{1} := \ker(\mathcal{L}_{f,\phi|\mathcal{C}^{r}}\, -\, \lambda_{f,\phi} I)$  and $E_{0} := \{\varphi \in \mathcal{C}^{r}(M , \C) : \int\varphi d\nu_{f,\phi} = 0 \}$. If $h_{f,\phi}$ is defined as the previous claim then $\int h_{f,\phi}d\nu_{f,\phi} = 1$ and $E_{1} = \{ t\cdot h_{f,\phi} : t \in \C \}$. Note that $E_{0}, E_{1}$ are $\mathcal{L}_{f,\phi|\mathcal{C}^{r}}-$invariant closed subspaces and $\mathcal{C}^{r}(M , \C) = E_{1} \oplus E_{0}$.\\

\emph{Claim 3: $spec(\tilde{\mathcal{L}}_{f,\phi|E_{0}}) \subset B(0 , \lambda_{1})$, where $0 < \lambda_{1} < 1$.}\\

\noindent In fact; endow $E_{0}$ of the $\mathcal{C}^{r}$ norm and take $\varphi \in E_{0}\cap \mathcal{C}^{r}(M , \R)$, with $||\varphi||_{r} = 1$. So $\varphi + 1 +\frac{1}{\kappa_{0}} \in \Lambda^{r}_{\kappa_{0}}$. By previous discussions, we know that $\ds\tilde{\mathcal{L}}^{n}_{f,\phi}( \varphi + 1 +\frac{1}{\kappa_{0}}) \xrightarrow[]{\mathcal{C}^{r}} \int(\varphi + 1+\frac{1}{\kappa_{0}})d\nu_{f,\phi} \cdot h_{f,\phi} = (1+\frac{1}{\kappa_{0}}) \cdot h_{f,\phi}$. Thus:
$$
||\tilde{\mathcal{L}}^{n}_{f,\phi}(\varphi)||_{r} = ||\tilde{\mathcal{L}}^{n}_{f,\phi}(\varphi + 1+\frac{1}{\kappa_{0}}) - \tilde{\mathcal{L}}^{n}_{f,\phi}(1+\frac{1}{\kappa_{0}})||_{r} \leq ||\tilde{\mathcal{L}}^{n}_{f,\phi}(\varphi + 1+\frac{1}{\kappa_{0}}) -(1+\frac{1}{\kappa_{0}}) \cdot h_{f,\phi}||_{r} +$$ $$||\tilde{\mathcal{L}}^{n}_{f,\phi}(1+\frac{1}{\kappa_{0}}) - (1+\frac{1}{\kappa_{0}}) \cdot h_{f,\phi}||_{r} \leq
\Delta \tau^{n-1}\big( \kappa_{0}(e^{\Delta} + 2) + e^{\Delta} \big) (3+\frac{2}{\kappa_{0}})||h_{f,\phi}||_{0}.
$$
Analogously, for $\varphi \in E_{0}$ complex function with $||\varphi||_{r} = 1$ we have $||\tilde{\mathcal{L}}^{n}_{f,\phi}(\varphi)||_{r} \leq \Delta \tau^{n-1}\big( \kappa_{0}(e^{\Delta} + 2) + e^{\Delta} \big) (3+\frac{2}{\kappa_{0}})||h_{f,\phi}||_{0}$.
Hence $\tilde{\mathcal{L}}^{n}_{f,\phi|E_{0}}$ is a contraction in the $\mathcal{C}^{r}$ norm for $n$ big enough and so $spec(\tilde{\mathcal{L}}_{f,\phi|E_{0}}) \subset B(0 , \lambda_{1})$, where $0 < \lambda_{1} < 1$. \\

The spectral gap property then follows from the previous claims.
\end{proof}

\begin{corollary}
Given $f \in D^{r}$ and $\phi \in C^{r}$ there exists a unique $\nu \in \ker(\mathcal{L}_{f , \phi|C^{r}}^{\ast} - \lambda_{f , \phi}I)$ such that $\nu(1) = 1$.
\end{corollary}
\begin{proof}
Let $\nu \in \ker(\mathcal{L}_{f , \phi|C^{r}}^{\ast} - \lambda_{f , \phi}I)$ be such that $\nu(1) = 1$. Given $\varphi \in C^{r}(M , \R)$, we know from the previous theorem that
$\tilde{\mathcal{L}}_{f , \phi}^{n}\varphi \rightarrow \int \varphi d\nu_{f,\phi} \cdot h_{f,\phi}$, hence
$$
\nu(\varphi) = \lim\nu(\tilde{\mathcal{L}}_{f , \phi}^{n}\varphi) = \nu(\int\varphi d\nu_{f,\phi} \cdot h_{f,\phi}) = \int\varphi d\nu_{f,\phi} \nu(h_{f,\phi}) =
\int\varphi d\nu_{f,\phi} \lim\nu(\tilde{\mathcal{L}}_{f , \phi}^{n}1) =
$$
$$
\int\varphi d\nu_{f,\phi}.
$$
 This prove the uniqueness.
\end{proof}

The next corollary assures us uniformity in the spectral gap. Let's remember that $\mathcal{D}^{r}\times C^{r}(M , \R) \ni (f , \phi) \longmapsto \frac{d\mu_{f,\phi}}{d\nu_{f,\phi}}$ is continuous, endowing the image with the usual uniform $C^0$ norm (see e.g. \cite{PU10}).
\begin{corollary}\label{pri}
Given $f \in D^{r}$ and $\phi \in C^{r}(M , \R)$ there exists neighborhoods $\mathcal{F}^{r}$ of $f$ and $\mathcal{W}^{r}$ of $\phi$,  constants $k \geq 0$ and $\tau \in (0 , 1)$ such that:  if $(\hat{f} , \hat{\phi}) \in \mathcal{F}^{r} \times\mathcal{W}^{r}$ then given $\varphi \in C^{r}(M , \R)$ we have
$$
||\tilde{\mathcal{L}}_{\hat{f} , \hat{\phi}} \varphi - \int\varphi d\nu_{\hat{f} , \hat{\phi}} \cdot h_{\hat{f},\hat{\phi}}||_{r} \leq k\tau^{n}||\varphi||_{r},
$$
for all $n \in \N$.
\end{corollary}
\begin{proof}
It follows directly from estimates of the \emph{Claim 3}, items (i) and (ii) of the previous theorem, as well as Corollary \ref{unigap}.
\end{proof}



\section{Proof of the main results}\label{Proofs}

\subsection{Linear response formula}

This section is devoted to the a proof of our Linear response formula result (Theorem~\ref{thm:linres}).

We already know that  statistical stability holds for expanding dynamics:

\begin{proposition}\label{exp3}
If $r \geq 1$, then the following maps are continuous:
\begin{itemize}
\item[(i)] $\mathcal{D}^{r}\times C^{r}(M , \R) \ni (f , \phi) \longmapsto P_{\topp}(f , \phi)$;
\item[(ii)] $\mathcal{D}^{r}\times C^{r}(M , \R) \ni (f , \phi) \longmapsto \frac{d\mu_{f,\phi}}{d\nu_{f,\phi}}$, endowing the image with the $C^0$-norm;
\item[(iii)] $\mathcal{D}^{r}\times C^{r}(M , \R) \ni (f , \phi) \longmapsto \nu_{f , \phi}$, endowing the image with the weak$^{\ast}$ topology;
\item[(iv)] $\mathcal{D}^{r}\times C^{r}(M , \R) \ni (f , \phi) \longmapsto \mu_{f , \phi}$, endowing the image with the weak$^{\ast}$ topology.
\end{itemize}
\end{proposition}

\begin{proof}
See e.g. \cite{PU10}.
\end{proof}

Now we will discuss the linear response formula results. We want to obtain differentiability results for $P_{\topp}(f,\phi), \frac{d\mu_{f,\phi}}{\nu_{f,\phi}}$ and $\mu_{f,\phi}$ with respect to $f$ and $\phi$.
Note that these elements
can be obtained by the spectral projection of $\mathcal{L}_{f,\phi}$ on the eigenspace  associated to the dominant eigenvalue.
Thereby, we will discuss about the differentiability of the Spectral Theory point
of view (for details about Spectral Theory see for instance \cite{K95}).

Let $A : E \rightarrow E$ be a bounded operator with the spectral gap property, let $\gamma$ a closed curve $C^{1}$ such that the bounded connected component
determined by it contains the dominant eigenvalue of $A$  and the  region which is exterior to the curve
 contains the  other spectral component of $A$. By semi continuity of the spectral components, there is a $\delta > 0$ such that if $\|A - \hat{A}\| < \delta$ then
$\gamma$ separates two spectral components of $\hat{A}$. We denote the spectral component of $\hat{A}$ which is limited by the connected component by $\Lambda_{\hat{A}}$. If we denote by
$P_{\hat{A}}$  the spectral projection associated to  the spectral component $\Lambda_{\hat{A}}$  of $\hat{A}$, by the holomorphic functional calculus we have
 that $P_{\hat{A}} = \frac{1}{2\pi i}\int_{\gamma}(zI - \hat{A})^{-1}dz$.
 As the curve $\gamma$ is fixed, the differentiability of $P_{\hat{A}}$ with respect to
$\hat{A}$, is the same of $(zI - \hat{A})^{-1}$ with respect to $\hat{A}$. It is known by the spectral theory that, in fact, $(zI - \hat{A})^{-1}$ is analytic with respect to $\hat{A}$ (we still have uniform convergence of Taylor series when we vary $z$ along of the curve $\gamma$) and therefore $P_{\hat{A}}$ is analytic with respect to $\hat{A}$.

Let us now return to the case of the transfer operator.

\begin{remark}
By \cite{Fr79} given a compact and connected smooth Riemannian manifold $M$, the space of maps $C^{r}(M , M)$ can be seen as a manifold modeled by the Banach space $C^{r}(M , \R^{2\mbox{dim}(M)})$. Furthermore, given $f \in C^{r}(M , M)$ its tangent space can be identified with $\Gamma^{r}_{f} := \{ \gamma \in C^{r}(M , TM) : \gamma(x) \in T_{f(x)}M, \forall x \in M\}.$
\end{remark}

Thus, the differentiability  of the spectral projection would imply the differentiability $\lambda_{f,\phi}$ and $\mu_{f,\phi}$ with respect the $f$ and $\phi$ if the transfer operators depended
smoothly with respect to dynamics and the potential. However, in general,
 the transfer operator depends smoothly (in fact, analytically) only on the potential(see \cite{PU10}  ). So, we have

\begin{proposition}\label{ana1}
Fixed $f \in \mathcal{D}^{r}$, the following maps are analytic:

 \begin{itemize}
  \item[i.] $C^{r}(M , \R) \ni \phi \mapsto \lambda_{f,\phi}$;
  \item[ii.] $C^{r}(M , \R) \ni \phi \mapsto h_{f,\phi} \in C^{r}(M , \R)$;
  \item[iii.] $C^{r}(M , \R) \ni \phi \mapsto \nu_{f,\phi} \in [C^{r}(M , \R)]^{\ast}$;
  \item[iv.] $C^{r}(M , \R) \ni \phi \mapsto \mu_{f,\phi} \in [C^{r}(M , \R)]^{\ast}$.
 \end{itemize}

\end{proposition}

 On the other hand, the transfer operator does not vary continuously with respect to the  dynamics, even in the expanding case, as an operator acting on $C^{r}(M , \R)$ (see \cite[Example 4.14]{CV13}). Nevertheless, in \cite{BCV16}, It is proved that the transfer operator is smooth with respect to smooth dynamics in a weaker sense, namely, requiring less regularity of its counter-domain. More accurately, there we have proven:

\begin{proposition} \label{ledif}
Let $r \geq 1$, $1 \leq k \leq r$. Let $\text{Diff}^{r}_{loc}(M, M)$ be the local $C^r$-diffeomorphisms space on a compact and connected Riemannian manifold $M$ and let $\phi\in C^{r}(M,\mathbb R)$ be some fixed potential.
Then the map
$$
\begin{array}{ccc}
\text{Diff}^{r}_{loc}(M)  & \to & L(C^{r}(M, \mathbb R), C^{r-k}(M,\mathbb R))\\
f &  \mapsto & \cL_{f, \phi}
\end{array}
$$
is $C^{k}$.
\end{proposition}

\begin{remark}
 It follows from the estimates done in the proof of the previous proposition that, in small neighborhoods of the local diffeomorphism, the derivative of $f \mapsto \cL_{f, \phi}$ is uniformly bounded.
\end{remark}

Therefore, we can not guarantee  that $\lambda_{f,\phi}$ and $\mu_{f,\phi}$ are smooth with respect to dynamics $f$ using the classical spectral theory, because it would only guarantee such differentiability if we had the differentiability of the transfer operator
 as operator acting on the same space.

For the reader convenience, we recall here the context of \cite[Theorem 8.1]{GL06} .
Let $r \geq 2$ be, $\mathcal{B}^{0} \supset \dots \supset \mathcal{B}^{r}$ be Banach spaces, $I$ be a Banach manifold and $\{A_{t}\}_{t \in I}$ be a bounded linear operators family acting
on the Banach spaces $\mathcal{B}^{i}$ such that $I \ni t \mapsto A_{t} \in L(\mathcal{B}^{1}, \mathcal{B}^{0})$ is continuous, where $L(\mathcal{B}^{i}, \mathcal{B}^{i-j})$ is the bounded linear operators space of $B^{i}$ on $B^{i-j}$. Furthermore, assume that
\begin{equation}\label{GL1}
\exists M > 0, \forall t \in I, \forall g \in \mathcal{B}^{0}, \; ||A_{t}^{n}g||_{\mathcal{B}^{0}} \leq C M^{n}||g||_{\mathcal{B}^{0}}
\end{equation}
and
\begin{equation}\label{GL2}
\exists \alpha < M, \forall t \in I, \; ||A_{t}^{n}g||_{\mathcal{B}^{1}} \leq C \alpha^{n}||g||_{\mathcal{B}^{1}} + C M^{n}||g||_{\mathcal{B}^{0}}.
\end{equation}
Assume also that for $j = 1, \ldots, r-1$, there exists the $j-$th derivative of the map $I \ni t \mapsto A_{t} \in L(\mathcal{B}^{r} , \mathcal{B}^{r - j})$. Denoting by $Q_{j}$ the
$j-$th derivative that acts from $I$ to $L(\mathcal{B}^{r} , \mathcal{B}^{r - j})$ , assume that for all $i \in [j,r]$ we have that $Q_{i}$ is bounded as a map of
$I$ in $L(\mathcal{B}^{i} , \mathcal{B}^{i - j})$. In these terms, It follows from\cite[Theorem 8.1]{GL06} that:

\begin{theorem}\label{dif1}
For $\varrho > \alpha$ and $\delta > 0$, denote by $V_{\varrho,\delta}$ the set of complex numbers $z$ such that $|z| \geq \varrho$ and, for all $1 \leq k \leq
r$, assume that the distance of $z$ to spectrum of $A_{t|\mathcal{B}^{k}}$ is bigger or equal than $\delta$. Then, the map
$I \times V_{\varrho,\delta} \ni (t,z) \mapsto (z - A_{t})^{-1} \in L(\mathcal{B}^{r} , \mathcal{B}^{0})$ is $C^{r-1}$ .
\end{theorem}

\begin{remark}\label{rem:restuny}
It follows also from \cite[Theorem 8.1]{GL06} that if $T_{i|(t,z)}$ denotes the $i$-th derivative of  $I \times V_{\varrho,\delta} \ni (t,z) \mapsto (z - A_{t})^{-1} \in L(\mathcal{B}^{r} , \mathcal{B}^{0})$ in $(t , z)$ then fixed $t \in I$ we have that
$$
\frac{||T_{i|(t+h_{1},z+h_{2})} - T_{i|(t,z)} - T_{i+1|(t , z)}\cdot (h_{1} , h_{2})||}{||(h_{1}, h_{2}||}
$$
 converges to 0, when $(h_{1} , h_{2})$ converges to 0, uniformly in $z \in V_{\varrho,\delta}$.
\end{remark}

Using this theorem and Corollary~\ref{pri}, about uniform spectral gap property, we can prove the Theorem~\ref{thm:linres}.

Before the proof of the theorem we need a strong stability result, namely:
\label{pri}
\begin{proposition}
The mapping $\mathcal{D}^{r} \ni f  \longmapsto  h_{f}= h_{f, \phi} \in \mathcal{C}^{r-1}(M , \R)$ is continuous.
\end{proposition}
\begin{proof}
Fix $f_{0} \in \mathcal{D}^{r}$. By Theorem~\ref{pri}, we have
$$
||\tilde{\mathcal{L}}^{n}_{f,\phi\label{pri} }1 - h_{f ,\phi}||_{0} \leq  k \tau^{n}, \forall f \in \mathcal{D}^{r}
.$$
Take $\epsilon > 0$, then there exists $n_{0} \in \N$ such that
$$
||\tilde{\mathcal{L}}^{n_{0}}_{f,\phi }(1) - h_{f ,\phi}|| < \frac{\epsilon}{3}, \forall f \in \mathcal{D}^{r}.
$$
 By proposition~\ref{exp3}, $\mathcal{D}^{r} \ni f \mapsto \lambda_{f,\phi}$ is continuous. Using the Proposition~\ref{ledif}, there exists a neighborhood $V$ of $f_{0}$ so that: if $f \in V$ we have that
$$
||\tilde{\mathcal{L}}^{n_{0}}_{f,\phi}(1) - \tilde{\mathcal{L}}^{n_{0}}_{f_{0},\phi}(1)||_{r-1} < \frac{\epsilon}{3}.
$$
Therefore, for $f \in V$:
$$
||h_{f,\phi } - h_{f_{0} ,\phi}||_{r-1} \leq ||h_{f ,\phi} -  \tilde{\mathcal{L}}^{n_{0}}_{f ,\phi}(1)||_{r-1} \;+ \;||\tilde{\mathcal{L}}^{n_{0}}_{f,\phi}(1) -
\tilde{\mathcal{L}}^{n_{0}}_{f_{0},\phi}(1)||_{r-1} + ||\tilde{\mathcal{L}}^{n_{0}}_{f_{0} ,\phi}(1) - h_{f_{0} ,\phi}||_{r-1} < \epsilon.
$$
This prove that $\mathcal{D}^{r} \ni f  \longmapsto  h_{f} \in \mathcal{C}^{r-1}(M , \R)$ is continuous in $f_{0}$.
\end{proof}

\begin{proof}[Proof of the Theorem~\ref{thm:linres}]

Fix $f_{0} \in \mathcal{D}^{r}$. Note initially that $\mathcal{L}_{f,\phi|C^{r-1}}$ and $\mathcal{L}_{f,\phi|C^{r}}$ have the same spectral radius, which is equal to  $e^{P_{\topp}(f,\phi)}$. By the previous proposition there exists a neighborhood $U$ of $f_{0}$ so that $\sup_{f \in U}||\mathcal{L}_{f,\phi}1||_{0} < +\infty$. By statistical stability (Proposition~\ref{exp3}) $\lambda_{f,\phi}$ is continuous. So, we can assume without loss of generality  that $\sup_{f \in U}\lambda_{f,\phi} < +\infty$. Furthermore,
there exists a closed $C^{1}$-curve $\gamma$ such that the bounded connected component determined by $\gamma$
contains the spectral radius of $\mathcal{L}_{f,\phi|C^{r}}$  for any $f \in U$, and the unbounded connected component contains the remainder of the spectrum  of $f$, for any $f \in U$.
Let $P_{f}$ be the spectral projection of $\mathcal{L}_{f,\phi|C^{r}}$ associated to its spectral radius. We already know that
$P_{f} = \frac{1}{2\pi i}\int_{\gamma}(zI - \mathcal{L}_{f,\phi})^{-1}dz$ for all $f \in U$. We will use the previous theorem to prove that
$U \times \gamma \ni (f, z) \mapsto (zI - \mathcal{L}_{f,\phi})^{-1} \in {L(C^{r} , C^{0})}$ is $C^{r-1}$ and that the error of the Taylor polynomial approximation of  order $r-1$  goes to 0 uniformly with respect to $z \in \gamma$, when we fix $f \in U$. Therefore, It is enough to choose correctly the elements contained in Theorem~\ref{dif1}. In the previous theorem take {$\mathcal{B}^{0} = C^{0}(M , \R)$, $\mathcal{B}^{1} = C^{1}(M , \R)$, $\mathcal{B}^{2} = C^{2}(M , \R), \ldots , \mathcal{B}^{r} = C^{r}(M , \R)$}, $I = U$, $t= f$, and so $A_{t} = A_{f}= \mathcal{L}_{f,\phi}$. By Proposition~\ref{ledif}, in order
 to apply Theorem~\ref{dif1}, It is enough to observe that the hypotheses (\ref{GL1}) and (\ref{GL2}) hold. By Theorem~\ref{pri}
{
$$
||\tilde{\mathcal{L}}^n_{f,\phi}g - \int g d\nu_{f,\phi} \cdot h_{f,\phi}||_{i} \leq k\tau^{n}, \forall g \in C^{i}(M, \R), \forall f \in U \text{ and } i = 1, \ldots, r.
$$
}
Take $M := \sup_{f \in U}||\mathcal{L}_{f,\phi}1||_{0} < +\infty,$ $\alpha := \sup_{f \in U}\lambda_{f,\phi} \cdot \tau$ and
$C := \max\{k , \sup_{f \in U}||h_{f,\phi}||_{r}\}$. Thereby,
$$
||\mathcal{L}_{f,\phi}^{n}g||_{0} \leq M^{n}||g||_{0}, \forall g \in C^{0}(M,\R) \text{ and } \forall f \in U,
$$
and
{
$$
\begin{array}{ccc}
||\mathcal{L}_{f,\phi}^{n}g||_{1} & \leq & k(\lambda_{f,\phi}\tau)^{n} + \lambda_{f,\phi}^{n}||g||_{1} \cdot||h_{f,\phi}||_{1} \\
  &\leq &
k(\sup_{f \in U}\lambda_{f,\phi} \cdot \tau)^{n} + \sup_{f \in U}||h_{f,\phi}||_{1}\cdot\sup_{f \in U}||\mathcal{L}_{f,\phi}1||_{0}^{n} \cdot ||g||_{1}.
\end{array}
$$
}
Thus, we can apply the previous theorem and  Remark~\ref{rem:restuny} to conclude that {$U \times \gamma \ni (f, z) \mapsto (zI - \mathcal{L}_{f,\phi})^{-1} \in L(C^{r} , C^{0})$ is $C^{r- 1}$} and
that the errors from the  approximation by the Taylor Polynomial 
of degree $r-1$ go to zero uniformly with respect to $z \in \gamma$, when we fix $f \in U$.
So, {$U \ni f \mapsto P_{f} \in L(C^{r} , C^{0})$} is $C^{r-1}$. As $\lambda_{f,\phi} = \frac{\mathcal{L}_{f,\phi}(P_{f}1)}{P_{f}1}, h_{f,\phi} = P_{f}1$ and $\nu_{f,\phi}(g) = \frac{P_{f}g}{P_{f}1}, \forall g \in C^{r}(M , \R),$ we prove the first part of the theorem.

Now we prove the second part of the theorem. We want to calculate an asymptotic formula for $D_{f,\phi}P_{\topp}(f,\phi)$. As it is already known that $\partial_{\phi}P_{\topp}(f,\phi) = \mu_{f,\phi}$ (see for example \cite{PU10}), It is enough to calculate $\partial_{f}P_{\topp}(f,\phi).$ The proof is completely analogous to the respective results in \cite{BCV16}
taking into account that the spectral gap holds for all smooth potential. Thereby we give a proof sketch.
We recall two key results of \cite{BCV16} (see Propositions~4.10 and ~4.11 therein):

\begin{lemma}{(Local Differentiability of inverse branches)}
Let $r \geq 1$, $0 \leq k \leq r$ and $f:M \to M$ be a $C^r$-local diffeomorphism on a compact connected manifold $M$.
Let $B= B(x, \delta)\subset M$ some ball such that the inverse branches $f_1, \dots, f_s : B \to M$ are
well defined diffeomorphisms onto their images. Then $C^r(M, M) \ni f \mapsto (f_1, \dots, f_s) \in C^{r-k}$ is a $C^{k}$ map.
\end{lemma}

We denote the derivative of $f \mapsto f_{j}$ by $T_{j|f}$.

\begin{lemma}
Let  $f \in {Diff}_{loc}^{r}$ and $\phi \in C^{r}(M , \R)$ be. Given $H \in \Gamma^{r}_{f},
g_{1} , g_{2} \in C^{r}(M, \R)$ and $t \in \R$ it holds
\begin{itemize}
\item[i)] $D_{f}(\mathcal{L}_{f,\phi}^{n}(g))_{|f_{0}} \cdot H
	= \sum_{i = 1}^{n}\mathcal{L}_{f_{0},\phi}^{i - 1} (D_{f}\mathcal{L}_{f,\phi}
	(\mathcal{L}_{f_{0},\phi}^{ n - i}(g))_{|f_{0}} \cdot H)$;
\item[ii)] there exists $c_{f,\phi}>0$ so that
	$\|D_{f}\mathcal{L}_{f,\phi}(g)_{|f_{0}} \cdot H\|_{0} \leq c_{f,\phi} \|g\|_{1} \; \|H\|_{1}$ and $c_{f,\phi}$ can be take uniform for dynamics close enough;
\item[iii)] $D_{f}\mathcal{L}_{f,\phi}^{n}(g_{1} + tg_{2})_{|f_{0}} \cdot H =   D_{f}\mathcal{L}_{f,\phi}^{n}(g_{1})_{|f_{0}} \cdot H + tD_{f}\mathcal{L}_{f,\phi}^{n}(g_{2})_{|f_{0}} \cdot H$.
\end{itemize}
\end{lemma}

We remark that $A \odot H$ will denote an operator $A$ acting on vector $H$. Fix $f_{0} \in \mathcal{D}^{r}$. By uniform spectral gap (Corollary~\ref{pri}) and the continuity of $\mathcal{D}^{r} \ni f \mapsto h$, there exists $K >0$ and neighborhood $W$ of $f_{0}$ so that $K^{-1} \leq \int \tilde{\mathcal{L}}^{n}_{f}1 d\nu_{f} \leq K$ for all $f \in W$ and $n \in \N$. In particular, $ \lim_{n\to\infty} \frac{1}{n}\log\int\mathcal{L}_{f}^{n} 1 \;
d\nu_{f_{0}}    = \log \lambda_{f} $ is a uniform limit with respect to $f\in W$. Consider the sequence of maps
$P_n: W\to \R$ given by
$
P_{n}(f)  = \frac{1}{n}\log\int\mathcal{L}_{f}^{n} 1\, d\nu_{f_{0}},
$
that converges uniformly to $\log \la_{f}$. We will see that the $D_{f}P_{n}$ converges uniformly. By chain rule:
 $$
  D_fP_{n|\hat{f}} \odot H    = \frac{\int D_{f }\mathcal{L}^{n}_{f }(1)_{|\hat{f} } \odot H
d\nu_{f_{0}} }{ n \cdot \int\mathcal{L}^{n}_{\hat{f}}(1)d\nu_{f_{0}}}    = \frac{ \int \sum_{i = 1}^{n}\mathcal{L}_{\hat{f}
}^{i - 1} (D_{f}\mathcal{L}_{f }(\mathcal{L}_{\hat{f}}^{ n - i}(1))_{|\hat{f} } \odot H) d\nu_{f_{0} } }{ n \cdot
\int\mathcal{L}^{n}_{\hat{f} }(1)d\nu_{f_{0}}}
  =
   $$
   $$
   \frac{B_{n}(\hat{f} , \hat{\phi}) \odot H}{\lambda_{\hat{f} }\int\tilde{\mathcal{L}}^{n}_{\hat{f}
}(1)d\nu_{f_{0} }} + \frac{C_{n}(\hat{f} , \hat{\phi}) \odot H}{\lambda_{\hat{f} }\int\tilde{\mathcal{L}}^{n}_{\hat{f}
}(1)d\nu_{f_{0} }},
 $$
 where $$
 B_{n}(\hat{f} ) \odot H = \frac{1}{n}\int \sum_{i =
1}^{n}\tilde{\mathcal{L}}_{\hat{f} }^{i - 1} (\sum_{j = 1}^{\deg(\hat{f})} e^{{\phi}(\hat{f}_{j}(\cdot))} \cdot
D\tilde{\mathcal{L}}_{\hat{f}}^{ n - i}(1)_{|\hat{f}_{j}(\cdot)} \odot [(T_{j|\hat{f}} \odot H)(\cdot)]) \; d\nu_{f_{0} } \text{ and }
$$
$$
C_{n}(\hat{f} ) \odot H     \! = \!\frac{1}{n}\int \sum_{i = 1}^{n}     \tilde{\mathcal{L}}_{\hat{f} }^{i - 1}  \Big(\!     \sum_{j
= 1}^{\deg(\hat{f})} e^{{\phi}(\hat{f}_{j}(\cdot))} \cdot   \tilde{\mathcal{L}}_{\hat{f} }^{ n - i}(1)(\hat{f}_{j}(\cdot)) \cdot
D{\phi}_{|\hat{f}_{j}(\cdot)} \odot [(T_{j|\hat{f}} \odot H)(\hat{f}_{j}(\cdot))] \Big)  d\nu_{f_{0} }.
$$
Analogously to the proof contained in \cite[Lemma~4.13]{BCV16}, taking into account the uniform spectral gap, we obtain:\\

\emph{Claim 1:} $B_{n}(\hat{f} )\odot H$ is uniformly convergent on $\hat{f} $ and $H \in \Gamma^{r}_{\hat{f}}$, $||H||_{r}
 \leq 1$, to the expression $\int \sum_{j = 1}^{\deg(\hat{f})} e^{{\phi}(\hat{f}_{j}(\cdot))} \cdot Dh_{\hat{f} ,
 \hat{\phi}|\hat{f}_{j}(\cdot)} \odot [(T_{j|\hat{f}} \odot H_{1})(\cdot)] d\nu_{\hat{f} } \cdot \int h_{\hat{f} } d\nu_{f_{0} }$.\\

\emph{Claim 2:} $C_{n}(\hat{f}) \cdot H$ is uniformly convergent on $\hat{f} $ and $H \in \Gamma^{r}_{\hat{f}}$,
 $||H||_{r}  \leq 1$, to the expression
$$
 \int \sum_{j = 1}^{\deg(\hat{f})} e^{{\phi}(\hat{f}_{j}(\cdot))} \cdot
h_{\hat{f} , {\phi}}(\hat{f}_{j}(\cdot)) \cdot D{\phi}_{|\hat{f}_{j}(\cdot)} \odot [(T_{j|\hat{f}} \odot H)(\hat{f}_{j}(\cdot))] \;
d\nu_{\hat{f} } \cdot \int h_{\hat{f} , {\phi}} d\nu_{f_{0}}.
$$ \vspace{.2cm}

As
$
D_{f}P_{n|\hat{f}} \odot (H) = \frac{B_{n}(\hat{f} , \hat{\phi}) \odot H}{\lambda_{\hat{f} }\int\tilde{\mathcal{L}}^{n}_{\hat{f}
}(1)d\nu_{f_{0} }} + \frac{C_{n}(\hat{f} , \hat{\phi}) \odot H}{\lambda_{\hat{f} }\int\tilde{\mathcal{L}}^{n}_{\hat{f}
}(1)d\nu_{f_{0} }},
$
 using the two claims and that $\int\tilde{\mathcal{L}}^{n}_{\hat{f} }(1)d\nu_{f_{0} }$ converges to $\int h_{\hat{f} ,
{\phi}}d\nu_{f_{0} }$ uniformly with respect to $\hat{f}$ we obtain that $D_{f}P_{n|\hat{f} } \odot H$ is uniformly convergent on
$\hat{f} $ and $H \in \Gamma^{r}_{\hat{f}}$, such that $||H||_{r} = 1$, to the sum
\begin{align*} &
\lambda_{\hat{f} ,{\phi}}^{-1} \int \sum_{j = 1}^{\deg(\hat{f})} e^{{\phi}(\hat{f}_{j}(\cdot))} \cdot Dh_{\hat{f}
{\phi}|\hat{f}_{j}(\cdot)} \odot [(T_{j|\hat{f}} \odot H_{1})(\cdot)] \; d\nu_{\hat{f} } \\  & + \lambda_{\hat{f} ,{\phi}}^{-1}
\int \sum_{j = 1}^{\deg(\hat{f})} e^{{\phi}(\hat{f}_{j}(\cdot))} \cdot h_{\hat{f} , {\phi}}(\hat{f}_{j}(\cdot)) \cdot
D{\phi}_{|\hat{f}_{j}(\cdot)} \odot [(T_{j|\hat{f}} \odot H)(\hat{f}_{j}(\cdot))] \; d\nu_{\hat{f}}. \end{align*}
This completes the proof of the theorem.
\end{proof}


\subsection{Large deviations stability \label{subseclarge}}

This section is devoted to the  proof of Theorem~\ref{thm:differentiability.LDP1}.

So far, we know that if $\psi$ is cohomologous to a constant then $\R \ni t \mapsto \cE_{f,\phi,\psi}(t)$ is affine; however, if $\psi$ not is cohomologous to a constant then
$\R \ni t \mapsto \cE_{f,\phi,\psi}(t)$ is real analytic and strictly convex. As a consequence of Theorem~\ref{thm:linres} we have:

\begin{proposition}\label{prop:free.energy}
Let $r \geq 2$. Then the following maps are $C^{r-1}$:
\begin{itemize}
\item[i.] $\R \times \mathcal{D}^{r} \times C^{r}(M, \R) \times C^{r}(M , \R) \ni (t,f,\phi,\psi) \mapsto \cE_{f,\phi,\psi}(t)$, \item[ii.] $\R \times \mathcal{D}^{r} \times C^{r}(M, \R) \times C^{r}(M , \R)  \ni (t,f,\phi,\psi) \mapsto \cE_{f,\phi,\psi}'(t)$.
\end{itemize}
\end{proposition}

Assume that $\psi$ is not  cohomologous to a constant and that $\int \psi \, d\mu_{f,\phi}=0$.
Using that $\R \ni t\to \cE_{f,\phi,\psi}(t)$ is strictly convex it is well defined its Legendre transform
$I_{f,\phi,\psi}$ by
\begin{equation*}
I_{f,\phi,\psi}(s)
    := \sup_{t \in \R} \; \big\{ st-\cE_{f,\phi,\psi}(t) \big\}.
\end{equation*}

This function is non-negative and strictly convex since $\cE_{f,\phi,\psi}$  is also strictly convex and
$I_{f,\phi,\psi}(s) = 0$ if and only if $s = \int \psi d\mu_{f,\phi}$. We can define the Legendre transform for a $\psi$ which is not cohomologous to a constant by
 $I_{f,\phi,\psi}(t) := I_{f,\phi, \psi - \int\psi d\mu_{f,\phi}}(t - \int\psi d\mu_{f,\phi}).$
Using the differentiability of the free energy function it is not hard to check the variational property
$$
I_{f,\phi,\psi} (\cE'_{f,\phi,\psi}(t) )
    = t \cE'_{f,\phi,\psi}(t) - \cE_{f,\phi,\psi}(t).
$$
Furthermore, $s \mapsto I_{f,\phi,\psi}(s)$ is strictly convex.

Using the differentiability of the free energy function, one can get results from large
deviations by Gartner-Ellis's theorem, with rate function given by the Legendre transform (see e.g. \cite{CRL98}). Our contribution  here is to prove the regularity of the large deviations rate function with respect to the expanding dynamics.

\begin{proof}[Proof of Theorem~\ref{thm:differentiability.LDP1}]
Let $\psi_{v_{\ast}}$ be an observable which is  not cohomologous to a constant. As  $\psi_{v}$ is cohomologous to a constant only if $\cE_{f,\phi,\psi}(t)$ is affine and by Theorem~\ref{thm:linres} there exists a neighborhood $U$ of $v_{\ast}$ so that $\psi_{v}$ is not cohomologous to a constant for all $v \in U$. Thus, we have proved the first part of the corollary.

Now we prove the second part of the corollary. Using the variational property of the Legendre transform and that
$\cE''_{f_v,\phi_v,\psi_v}(t) > 0$, we have that for all
$$s \in (\inf_{t \in \R}\int \psi_{v} d\mu_{\phi_{v} + t\psi_{v}} \; , \sup_{t \in \R}\int \psi_{v} d\mu_{\phi_{v} + t\psi_{v}})$$ there exists a unique $t=t(s,v)$
such that $s=\cE'_{f_v,\phi_v,\psi_v}(t) $ and
\begin{equation}\label{eq.var.rate}
I_{f_v,\phi_v,\psi_v} (s)= s \cdot t(s, v) - \cE_{f_v,\phi_v,\psi_v}(t(s, v)).
\end{equation}
By taking a smaller neighborhood $U$ of $v_{\ast}$ if necessary,  there exists an open interval $J$ with the equation (\ref{eq.var.rate}) holds for $s \in \overline{J}$ and $v \in U$.
Now, consider  the skew-product
$$
\begin{array}{ccc}
F: V \times \overline{J} & \to & V \times \R\\
(v,t) & \mapsto & (v, \cE'_{f_v,\phi_v,\psi_v}(t)).
\end{array}
$$
It is injective because it is strictly decreasing along the fibers.
As $V\times \overline{J}$ is a compact metric space, then $F$ is a homeomorphism on the image $F(V\times \overline{J})$. In fact it is a diffeomorphism $C^{r-1}$. Thus, applying the Implicit Function Theorem we have that for all $(v,s)\in F(V\times \overline{J})$ there exists a unique $t=t(v,s)$, depending $C^{r-1}$ with respect to $(v,s)$, so such $F(v,t(v,s))=(v, s)$ and
$s=\cE'_{f_v,\phi_v,\psi_v}(t)$. Then, It follows from equation \eqref{eq.var.rate} that $(v, s) \mapsto I_{f_v,\phi_v,\psi_v}(s)$ is $C^{r-1}$.
\end{proof}


\subsection{Stability of the statistical laws}

This subsection is devoted to the a proof of  Theorem \ref{thm:stalaw1} and Corollary \ref{coho1}.

Before we prove Theorem~\ref{thm:stalaw1}, we  prove a version of  part of the Livsic's theorem. Such theorem 
ensures regularity of the solution of a cohomological equation
in the context of a  hyperbolic dynamics. 

\begin{proposition} \label{propcohomo}
Let $f \in \mathcal{D}^{r}$ be and $\psi \in C^{r}(M, \R)$. If $S_{n}\psi(p) = 0$ for all $p$ such that $f^{n}(p) = p$ then there exists $u \in C^{r}(M , \R)$ such that $\psi = u\circ f - u$.
\end{proposition}
\begin{proof}
This result is well known for diffeomorphisms (see  \cite{KKPW89}). 

So, for the expanding maps case, let us write down an independent proof.

By Proposition 3.4.5 in \cite{PU10}, there exists a H\"older function $u \in C^{\eta}(M, R)$ such that 

\begin{equation}
\psi= u \circ f - u \Rightarrow e^\psi \cdot e^u = e^{u \circ f}. \label{eqhom}
\end{equation}

Let us consider the transfer operator $\mathcal{L}_{f, \psi}$. It is classical fact that It has spectral gap as an operator acting 
in the space of H\"older functions. As we have seen in the previous section of this paper, it  has also spectral gap as an operator acting in $C^r$. 

Using equation \ref{eqhom}, we have  that
$$
\mathcal{L}_{f, \psi}(e^u)(x)= \sum_{y \in f^{-1}(x)} e^{\psi(y)} \cdot e^{u(y)}= \sum_{y \in f^{-1}(x)} e^{u \circ f(y)}= e^{u(x)} \cdot \deg(f) \Rightarrow \mathcal{L}^n_{f, \psi}(e^u)= e^u \deg(f)^n.
$$
Since $\psi$ is $\eta-$H\"older, by \cite{PU10} $\mathcal{L}_{f, \psi}$ has the spectral gap property as an operator acting in $C^{\eta}(M, R)$. Therefore 
$\frac {\mathcal{L}_{f, \psi}}{\rho(\psi)^n }(e^u)$ converges 
uniformly to $\int_M e^u d \nu_\psi \cdot h_\psi$, 
where $\nu_\psi$, $\rho(\psi)$, $h_\psi$,  are respectively,  the conformal probability measure,  the dominant eigenvalue and its corresponding normalized eigenvector, all of them, associated to  of $\mathcal{L}_{f, \psi}$. Note that $h_\psi \in C^{\eta}(M, R)$ is 
strictly positive,  and $\int_M e^u d\nu_\psi \geq \inf e^u> 0$.

Since $\psi \in C^r$, we also have that $\mathcal{L}_{f, \psi}|_{C^r}$ has spectral gap property, with the same eigenvalue  $\rho(\psi)$ and normalized eigenvector $h_\psi$, which is $C^r$.

Note that, since
$
\frac{e^u \cdot \deg(f) ^n}{\rho(\psi) ^n}= \frac{\mathcal{L}^n_{\psi}(e^u)}{\rho(\psi) ^n}  
$
converges to $\int_M e^u d\nu_\psi \cdot h_\psi$,
we conclude that 
$\deg(f) = \rho(\psi)$ and $e^{u}= \int_M e^u d\nu_\psi \cdot h_\psi$,
and so
$$
u = \log\Big(\int_M e^u d\nu_\psi \Big )+ \log( h_\psi) \in C^r.
$$

\end{proof}

\begin{proof}[Proof of Theorem~\ref{thm:stalaw1}]

We prove first that $\psi$ is cohomologous to a constant in $L^{2}(\mu_{f,\phi})$ if and only if, $\psi$ is cohomologous to a constant in $C^{r}(M , \R)$. Thus the first part of the corollary follows from the already known Central Limit Theorem  for expanding dynamics (see e.g. \cite{PU10}). Let then $\psi$ cohomologous to a constant in $L^{2}(\mu_{f,\phi})$. Note that
$$
\begin{array}{ccc}
\mathcal{E}_{f,\phi,\psi}''(0) & = & \lim_{n \to +\infty}\frac{1}{n}( \int (S_{n}\psi)^{2} d\mu_{f,\phi} - (\int S_{n}\psi d\mu_{f,\phi})^{2}) \\
                               & = & \lim_{n \to +\infty}\frac{1}{n}\int(S_{n}\tilde{\psi})^{2} d\mu_{f,\phi} \\
                               & = & \int \tilde{\psi}^{2} d\mu_{f,\phi} + 2\lim_{n \to +\infty}\int \sum_{j=1}^{n-1}\Big(1 - \frac{j}{n}\Big)
\tilde{\psi} \circ f^{j}\tilde{\psi} d\mu_{f,\phi} \\
                               & = & \int \tilde{\psi}^{2} d\mu_{f,\phi} + 2\sum_{j=1}^{\infty}\int \tilde{\psi} \circ f^{j}\tilde{\psi} d\mu_{f,\phi} \\
                              & = & \sigma^{2}_{f,\phi}(\psi).
\end{array}
$$
By the dichotomy of the Central Limit Theorem, we have that $\mathcal{E}_{f,\phi,\psi}''(0) = 0$ and thus $\mathcal{E}_{f,\phi,\psi}$ is not strictly convex. So $\psi$ is cohomologous to a constant.
Hence $S_{n}\psi(x) = n\int \psi d\mu_{f,\phi}$ for all periodic point
$x$ of period $n$. Thus, using the previous proposition and conclude that $\psi = u\circ f - u+\int \psi
\,d\mu_{f,\phi}$ for some function $u \in \mathcal{C}^{r}(M , \R)$. So we proved the first part of the theorem.

Now we prove the second part of the theorem. The regularity of $\mathcal{D}^{r}\times C^{r}(M, \R) \times C^{r}(M, \R) \ni(f,\phi,\psi) \mapsto m_{f,\phi}(\psi)$ is a direct application of the Theorem~\ref{thm:linres}. Finally, fix $(f_{0} , \phi_{0}) \in \mathcal{D}^{r}\times C^{r}(M, \R)$. For each $(f , \phi) \in \mathcal{D}^{r}\times C^{r}(M, \R)$
define $E_{0,f,\phi} := \ker \nu_{f,\phi}
\cap C^{r}(M , \R)$, and $T_{f,\phi}g = (g - \int gd\mu_{f,\phi}) \cdot h_{f,\phi}$ for all $g \in C^{r}(M , \R)$. So $T_{f,\phi}$ it is a bounded operator whose restriction  to
$E_{0,f_{0},\phi_{0}}$ is a linear isomorphism on $E_{0,f,\phi}$, with $T_{f,\phi}^{-1}g = \frac{g}{h_{f,\phi}} - \int \frac{g}{h_{f,\phi}}d\nu_{f_{0},\phi_{0}}$.
Given $\psi \in C^{r}(M , \R)$ we already know that $\sigma_{f,\phi}^{2}(\psi)
        = \int \tilde \psi^{2}\;d\mu_{f,\phi} + 2\sum_{n = 1}^{\infty}\int \tilde \psi(\tilde \psi\circ f^{n}) \;d\mu_{f,\phi}$, where
        $\tilde \psi=\psi-\int\psi \;d\mu_{f,\phi}$. By a simple computation one obtains
$$
\int \tilde \psi(\tilde \psi\circ f^{n}) \;d\mu_{f,\phi} = \int\tilde{\psi} \tilde{\mathcal{L}}^{n}_{f,\phi}(\tilde{\psi} \cdot h_{f,\phi})\;d\nu_{f,\phi} =
\int\tilde{\psi} T_{f,\phi}^{-1}\tilde{\mathcal{L}}^{n}_{f,\phi}T_{f,\phi}T_{f,\phi}^{-1}(\tilde{\psi}\cdot h_{f,\phi}) \;d\mu_{f,\phi}.
$$
Thus
$$
\sigma_{f,\phi}^{2}(\psi) = -\int \tilde \psi^{2}\;d\mu_{f,\phi} + 2 \int \tilde \psi (I - T_{f,\phi}^{-1}\tilde{\mathcal{L}}^{n}_{f,\phi}T_{f,\phi|E_{0,f_{0},\phi_{0}}})^{-1}
T_{f,\phi}^{-1}(\tilde{\psi} \cdot h_{f,\phi})\;d\mu_{f.\phi}.
$$
Using the same ideas contained in the proof of the Theorem \ref{thm:linres}, one can apply the Theorem~\ref{dif1} concluding that
$\mathcal{D}^{r}\times C^{r}(M, \R)\ni (f,\phi) \mapsto (I - T_{f,\phi}^{-1}\tilde{\mathcal{L}}^{n}_{f,\phi}T_{f,\phi|E_{0,f_{0},\phi_{0}}})^{-1}$ is $C^{r-1}$.
This completes the proof of the theorem.

\end{proof}

\begin{proof}[Proof of Corollary~\ref{coho1}]
In the proof of the previous theorem, we proved that $\psi$ is cohomologous to a constant in $L^{2}(\mu_{f,\phi})$ if, and only if, $\psi$ is cohomologous to a constant in $C^{r}$. Furthermore, $\psi$ is cohomologous to a constant in $L^{2}(\mu_{f,\phi})$ if, and only if, $\sigma^2_{f,\phi}(\psi) =0$, by previous theorem. Thus the corollary follows from the continuity of $\sigma^2_{f,\phi}(\psi)$ with respect  to $f, \phi$ and $\psi$, which is guaranteed by Theorem~\ref{thm:linres}.
\end{proof}


\subsection{Multifractal analysis}

This subsection is devoted to the proof of Corollary~\ref{corexp}.

\begin{proof}[Proof of Corollary~\ref{corexp}]

From what we have already commented in section~\ref{sec:multanal}, the corollary is a direct consequence of Theorem~\ref{thm:differentiability.LDP1} since we prove that $Y:= \{(v,c) \in V \times \R^+_0 : c < \sup_{\eta \in \mathcal{M}_{1}(f_{v})}|\int \psi_{v} d\mu_{f_{v},\phi_{v}} - \int \psi_{v} d\eta|\}$ is an open set. As we already know that $v \mapsto \mu_{f_{v},\phi_{v}}$ is continuous, to prove that $Y$ is open it is enough to prove that:  fixed
$v_{\ast} \in V$ and $\eta \in \mathcal{M}_{1}(f_{v_{\ast}})$, given $\epsilon > 0$ there exists a neighborhood $\hat{V}$ of $v_{\ast}$ such that if $v\in \hat{V}$ then there exists $\eta_{v} \in \mathcal{M}_{1}(f_{v})$
with $|\int \psi_{v_{\ast}}d\eta_{v} - \int\psi_{v_{\ast}} d\eta| \leq \epsilon$.

Let us prove this fact. Fix $v_{\ast} \in V$. Given $\epsilon > 0$,
by specification property of expanding dynamics, there exists $N(\epsilon)$ so that 
any finite sequence of finite pieces of  orbit of $f_{v_{\ast}}$ can be shadowed by a point $x \in M$
using $f_{v_{\ast}}$, except possibily for gaps of $N(\epsilon)$ iterates between the pieces of orbit. In general $N(\epsilon)$ depends on the dynamics. However,  as our dynamics are expanding, as long as we choose a small neighborhood $V_{1}$ of $v_{\ast}$, we can assume that $N(\epsilon)$ is the same for all dynamics $f_{v}$,
$v \in V_{1}$.  By Birkhoff's ergodic theorem,
there exists a $x \in M$ and $n > 4\frac{N(\epsilon)}{\epsilon}$ such that
$
|\frac{1}{n}\sum_{i = 0}^{n-1}\psi_{v_{\ast}}(f^{i}_{v_{\ast}}(x)) - \int \psi_{v_{\ast}}
d\eta| < \frac{\epsilon}{3}.
$
As the parametrization is continuous there exists a neighborhood
$\hat{V} \subset V_{1}$ of $v_{\ast}$ so that $d(f^{j}_{v_{\ast}}(y) , f^{j}_{v}(y)) < \frac{\epsilon}{3||\psi||_{1}},$ for all $j \in \{1 ,\ldots, n-1\}$, $ y \in M$ and $v \in \hat{V}$.
Thus, for each $v \in \hat{V}$ and $l \in \N$ there exists a $y_{v,l} \in M$ such that when iterate $y_{v,l}$ by $f_{v}$ we shadow  $l$ times the orbit piece $x, \ldots, f_{v}^{n-1}(x)$, with a bounded error in each jump
of  $N(\epsilon)$ iterates.
  Taking an accumulation point of the sequence of probabilities $\Big(\frac{1}{l(n + N(\epsilon)-1)}\sum_{i = 0}^{l(n + N(\epsilon)-1) - 1}\delta_{{f^{i}_{v}} (y_{v,l})} \Big)_{l \in \N}$ we obtain an  $f_{v}-$invariant probability $\eta_{v}$
with $|\int \psi_{v_{\ast}}d\eta_{v} - \int\psi_{v_{\ast}} d\eta| \leq \epsilon$ and therefore we conclude that $Y$ is open.
This completes the proof of the corollary.
\end{proof}


\vspace{.3cm}
\subsection*{Acknowledgements.}
This work was partially supported by CNPq and Capes and is
part of the first author's PhD thesis at Federal University of Bahia. The authors are deeply
grateful to P. Varandas for useful comments.

\bibliographystyle{alpha}

\end{document}